\documentclass[reqno]{amsart}
\usepackage{amsmath, amssymb, amsthm, epsfig}
\usepackage{hyperref, latexsym}
\usepackage{url}
\usepackage[mathscr]{euscript}
\usepackage{mathabx}
\usepackage{color}
\usepackage{fullpage} 
\usepackage{setspace}

\allowdisplaybreaks

\def\today{\ifcase\month\or
  January\or February\or March\or April\or May\or June\or
  July\or August\or September\or October\or November\or December\fi
  \space\number\day, \number\year}

\DeclareMathOperator{\supp}{\mathrm{supp}}

 \newtheorem{theorem}{Theorem}
  
 \newtheorem{lemma}[theorem]{Lemma}
 \newtheorem{proposition}[theorem]{Proposition}
 \newtheorem{corollary}[theorem]{Corollary}
 \theoremstyle{definition}

 \theoremstyle{remark}

 \newcommand{\mc}{\mathcal}

 \newcommand{\C}{\mathbb{C}}
 \newcommand{\R}{\mathbb{R}}
 \newcommand{\N}{\mathbb{N}}
 
 \newcommand{\Z}{\mathbb{Z}}
 \newcommand{\csch}{\mathrm{csch}}

 \newcommand{\dt}{\text{\rm d}t}
  \renewcommand{\d}{\text{\rm d}}

 \newcommand{\dx}{\text{\rm d}x}

 \newcommand{\dmu}{\text{\rm d}\mu}

\begin{document}
\title[A tale of three integrals]{On Montgomery's pair correlation conjecture:\\ a tale of three integrals}
\author[Carneiro, Chandee, Chirre, and Milinovich]{Emanuel Carneiro, Vorrapan Chandee, Andr\'{e}s Chirre and Micah B. Milinovich}
 %, bandlimited approximations}

\address{ ICTP - The Abdus Salam International Centre for Theoretical Physics, Strada Costiera, 11, I - 34151, Trieste, Italy.}

\email{carneiro@ictp.it}

\address{Mathematics Department, Kansas State University, 125 Cardwell Hall, Manhattan, KS, 66503 USA.}

\email{chandee@ksu.edu}

\address{Department of Mathematical Sciences, Norwegian University of Science and Technology, NO-7491 Trondheim, Norway.}

\email{carlos.a.c.chavez@ntnu.no }

\address{Department of Mathematics, University of Mississippi, University, MS 38677 USA.}

\email{mbmilino@olemiss.edu}

\allowdisplaybreaks
\numberwithin{equation}{section}

\maketitle  

\begin{abstract}
We study three integrals related to the celebrated pair correlation conjecture of H.~L.~ Montgomery. The first is the integral of Montgomery's function $F(\alpha, T)$ in bounded intervals, the second is an integral introduced by Selberg related to estimating the variance of primes in short intervals, and the last is the second moment of the logarithmic derivative of the Riemann zeta-function near the critical line. The conjectured asymptotic for any of these three integrals is equivalent to Montgomery's pair correlation conjecture. Assuming the Riemann hypothesis, we substantially improve the known upper and lower bounds for these integrals by introducing new connections to certain extremal problems in Fourier analysis. In an appendix, we study the intriguing problem of establishing the sharp form of an embedding between two Hilbert spaces of entire functions naturally connected to Montgomery's pair correlation conjecture.
\end{abstract}

%\tableofcontents

\section{Introduction}
\subsection{Background} Let $\zeta(s)$ denote the Riemann zeta-function and let 
\[
\displaystyle \psi(x) = \sum_{n\le x} \Lambda(n),
\]
where $\Lambda(n) = \log p$ if $n=p^k$ for a prime $p$ and $k\in\mathbb{N}$, and $\Lambda(n)=0$ otherwise. In order to study the distribution of primes in short intervals, Selberg \cite{S} introduced the integrals
\[
I(a,T) := \int_1^T \left| \frac{\zeta'}{\zeta}\left( \frac{1}{2}+\frac{a}{\log T}+it\right) \right|^2  \mathrm{d}t
\]
for $a>0$ and
\begin{equation}\label{20200918_12:11}
J(\beta,T) := \int_1^{T^\beta} \left(\psi\!\left(x+\frac{x}{T}\right)-\psi(x)-\frac{x}{T} \right)^{\!2} \frac{\mathrm{d}x}{x^2}
\end{equation}
for $\beta\ge0$. For $0\le \beta \le 1$, Gallagher and Mueller \cite{GaMu} proved that
\begin{equation} \label{asymp}
J(\beta,T) \sim \frac{\beta^2}{2} \frac{\log^2T}{T}, \quad \text{as } T\to \infty.
\end{equation}
Assuming the Riemann hypothesis (RH), Selberg  \cite{S} proved an upper bound for $I(a,T)$ when $a \ge 10$ and used this to show that
\begin{equation}\label{SelbergBound}
J(\beta,T) = O_\beta\!\left( \frac{\log^2 T}{T} \right), \quad \text{as } T\to \infty,
\end{equation}
for $1 < \beta \le 4$. Selberg's proof can be modified to show that the estimate in \eqref{SelbergBound} holds for each fixed $\beta > 1$. Assuming RH, for each $\beta > 1$, it is now known that there are constants $D^{\pm}$ such that
\begin{equation}\label{MontgomeryBound}
\big(D^-\beta + o(1)\big)\cdot \frac{\log^2 T}{T} \le J(\beta,T) \le \big(D^+ \beta+ o(1)\big) \cdot \frac{\log^2 T}{T},
\end{equation}
as $T \to \infty$. In particular, we see that the dependence on the parameter $\beta$ is linear. The proof of the upper bound in this form was first given by Montgomery (unpublished) while alternate proofs have been given in \cite{GaMu,GG,GGM,GM}. The proof of the lower bound is due to Goldston and Gonek \cite{GG}. 

\subsection{Equivalences to Montgomery's pair correlation conjecture} In order to study the pair correlation of the zeros of $\zeta(s)$, for $\alpha \in \mathbb{R}$ and $T\ge 2$, Montgomery \cite{M} introduced the form factor
\[
F(\alpha):=F(\alpha,T) = \frac{2\pi}{T\log T} \sum_{0<\gamma,\gamma'\le T} T^{i \alpha (\gamma-\gamma')} w(\gamma-\gamma'),
\] 
where $w(u)=4/(4+u^2)$. Here the double sum runs over the ordinates $\gamma,\gamma'$ of two sets of non-trivial zeros of $\zeta(s)$, counted with multiplicity. We use the shorthand notation $F(\alpha)$ for simplicity, but the reader should always keep in mind that this is also a function of the parameter $T$. 
It follows from the definition that  $F(\alpha)$ is even and real-valued. Moreover, since
\[
\sum_{0<\gamma,\gamma'\le T} T^{i \alpha (\gamma-\gamma')} w(\gamma-\gamma') = 2 \pi \int_{-\infty}^\infty e^{-4 \pi |u|} \bigg| \sum_{0<\gamma\le T} T^{i \alpha \gamma} e^{2\pi i \gamma u} \bigg|^2 \mathrm{d}u, 
\]
it follows that $F(\alpha) \ge 0$ for all $\alpha \in \mathbb{R}$. Montgomery was interested in the asymptotic behavior of the function $F(\alpha)$ since, by Fourier inversion, we have
\begin{equation}\label{inversion}
 \sum_{0<\gamma,\gamma'\le T} R\!\left((\gamma-\gamma') \frac{\log T}{2\pi} \right) w(\gamma-\gamma') = \frac{T\log T}{2\pi} \int_{-\infty}^\infty\widehat{R}(\alpha) \, F(\alpha) \, \mathrm{d}\alpha
\end{equation}
for any function $R \in L^1(\mathbb{R})$ such that $\widehat{R} \in L^1(\mathbb{R})$, where
\[
\widehat{R}(\alpha) = \int_{-\infty}^\infty  e^{-2\pi i \alpha x} \, R(x)\,\mathrm{d}x
\]
denotes the usual Fourier transform of $R$. Assuming RH, it is known that
\begin{equation}\label{F formula}
F(\alpha,T) = \Big(T^{-2|\alpha|}\log T + |\alpha| \Big) \left( 1 + O\!\left( \sqrt{\frac{\log\log T}{\log T}} \right)\right), \quad \text{as } T\to \infty, 
\end{equation}
uniformly for $0\le |\alpha| \le 1$. This was proved by Goldston and Montgomery \cite[Lemma 8]{GM}, refining the original work of Montgomery \cite{M}. This asymptotic formula allows one to estimate the sum on the left-hand side of $\eqref{inversion}$ for $R \in L^1(\mathbb{R})$ with $\mathrm{supp}(\widehat{R}) \subset [-1,1]$. Montgomery conjectured that $F(\alpha)\sim 1$ for $|\alpha|>1$, uniformly for $\alpha$ in bounded intervals. This is sometimes called Montgomery's strong pair correlation conjecture. This assumption, via approximating the characteristic function of an interval by bandlimited functions, led Montgomery to further conjecture that, for any fixed $\beta>0$,
\smallskip
\begin{enumerate}
\item[\textup{(I)}] \ $\displaystyle N(\beta,T):=\!\!\!\sum_{\substack{ 0<\gamma,\gamma'\le T \\ 0<\gamma-\gamma' \le \frac{2\pi \beta}{\log T} }} 1 \ \sim \  \frac{T \log T}{2\pi}  \int_0^\beta \left\{ 1 - \Big( \frac{\sin \pi u}{\pi u}\Big)^2 \right\} \, \mathrm{d}u, \quad \text{as } T\to \infty. $
\end{enumerate}
\smallskip
This is known as Montgomery’s pair correlation conjecture. Since there are $\displaystyle \sim T\log T/(2\pi)$ non-trivial zeros of $\zeta(s)$ with ordinates in the interval $(0,T]$ as $T\to \infty$, the function $N(\beta,T)$ counts the number of pairs of zeros within $\beta$ times the average spacing between zeros. 

\smallskip

Assuming RH, from the works of Gallagher and Mueller \cite{GaMu}, Goldston \cite{G}, and Goldston, Gonek and Montgomery \cite{GGM}, it is known that the following asymptotic formulae are equivalent to the validity of Montgomery's pair correlation conjecture in (I) for each fixed $\beta>0$:

\begin{enumerate}
	
\item[\textup{(II)}] $\displaystyle \int_{b}^{b+\ell}F(\alpha,T) \,\d\alpha  \sim  \ell, \quad \text{as } T\to \infty \text{ for any fixed} \  b\ge 1 \ {\rm and} \ \ell > 0$;

\medskip

\item[\textup{(III)}] \ $\displaystyle J(\beta,T) \sim  \left(\beta-\frac{1}{2}\right)  \frac{\log^2T}{T}, \quad \text{as } T\to \infty \text{ for any fixed } \beta > 1$;

\medskip

\item[\textup{(IV)}] \ $\displaystyle I(a,T) \sim  \left( \frac{1-e^{-2a}}{4 a^2} \right) T \log^2 T, \!\quad \text{as } T\to \infty \text{ for any fixed } a>0$.

\smallskip
\end{enumerate}

Since Montgomery's pair correlation conjecture remains a difficult open problem, it is natural to instead ask for upper and lower bounds for the functions $N(\beta,T)$, $\int_{b}^{b+\ell}F(\alpha,T)\,\d\alpha$, $J(\beta,T)$, and $I(a,T)$ in place of asymptotic formulae. Assuming RH, extending previous work of Gallagher \cite{Ga}, it was shown in \cite{CCLM} that
\[
N(T) \left( \beta - \frac{7}{6} + \frac{1}{2\pi^2\beta} + O\left(\frac{1}{\beta^2}\right) + o(1) \right) \, \le \, N(\beta,T) \, \le \, N(T) \left( \beta+ \frac{1}{2\pi^2\beta} + O\left(\frac{1}{\beta^2}\right) + o(1) \right),
\]
as $T \to \infty$, for all $\beta>0$, by using \eqref{inversion}, \eqref{F formula}, and certain extremal functions of exponential type. Here $N(T)$ denotes the number of non-trivial zeros of $\zeta(s)$ with ordinates in the interval $(0,T]$, and the term $7/6$ in the lower bound can be replaced by 1 if we further assume that almost all zeros of $\zeta(s)$ are simple. 

\smallskip

The purpose of this paper is to continue this direction of investigation and, using tools from Fourier analysis, substantially improve the current upper and lower bounds for the integrals in (II), (III), and (IV) assuming RH. As we shall see, novel insights and certain Fourier optimization problems emerge when we treat each of these integrals.

\subsection{Summary of results}\label{Sec_summary} We now present an overview of some of our main results. Theorems \ref{Thm1_20201215*} and \ref{Thm_20201215_01:27*} below (and their corollaries) are representatives of a much more detailed discussion that follows in Sections \ref{Sec2_integral_F} and \ref{Sec_PSI_new}, respectively. These sample results already give a clear perspective of the magnitude of the improvements in this paper over previous results.

\subsubsection{The integral of $F(\alpha)$ in bounded intervals} An important feature of this paper is the development of a general theoretical framework relating the objects we want to bound in analytic number theory to certain extremal problems in Fourier analysis. For some of these extremal problems, achieving the exact answer is a hard task, and we must rely on certain test configurations to provide reasonable approximations. For instance, we define universal constants ${\bf C}^+$ and ${\bf C}^-$ in \textsection \ref{PfThm1_UB} and  \textsection \ref{PfThm1_LB} %\eqref{20210112_11:26} and \eqref{20210204_11:03} below, 
as solutions of two such extremal problems, and use them to prove the following theorem. 
\begin{theorem}\label{Thm1_20201215*}
Assume RH, let $b \geq 1$, and let $\varepsilon >0$ be an arbitrary number. For large $\ell$, as $T \to \infty$, we have
\begin{equation*}%\label{20201215_00:21*}
 ({\bf C}^- - \varepsilon)\,\ell +o(1)
\le  \int_b^{b+\ell}  F(\alpha,T) \, \d\alpha \le  ({\bf C}^+ + \varepsilon)\,\ell +o(1),
\end{equation*}
where the constants ${\bf C}^+$ and ${\bf C}^-$ are defined in \eqref{20210112_11:26} and \eqref{20210204_11:03}, respectively. 
\end{theorem}

We establish the bounds 
\begin{equation}\label{20210205_11:56}
0.9278 < {\bf C}^- \leq  {\bf C}^+ < 1.3302
\end{equation}
for these universal constants, which immediately leads to the following corollary. 

%\newb{Here is a typical theorem in our framework, and the corresponding corollary arising from \eqref{20210205_11:56}.}

\begin{corollary}\label{Thm1_20201215}
Assume RH and let $b \geq 1$. For large $\ell$, as $T \to \infty$, we have
\begin{equation}\label{20201215_00:21}
0.9278 \, \ell + o(1) 
\le  \int_b^{b+\ell}  F(\alpha,T) \, \d\alpha  \leq 1.3302\,\ell + o(1).
\end{equation}
\end{corollary}

We use this theorem to give information about the distribution of primes in short intervals. Furthermore, the work of Radziwi\l\l \, \cite{R} illustrates a connection between Theorem \ref{Thm1_20201215*} and the theoretical limitations of mollifying the Riemann zeta-function on the critical line (see \textsection \ref{Rad_sec}). Previously, the best known bounds in \eqref{20201215_00:21} were due to Goldston \cite[Lemma A]{G2} and Goldston and Gonek \cite[Lemma]{GG}, respectively, where an estimate with $\frac{1}{3}$ in place of $0.9278$ in the lower bound and $2$ in place of $1.3302$ in the upper bound can be established for sufficiently large $\ell$ by adding up integrals of length 2.

\smallskip

Theorem \ref{Thm1_20201215*} and Corollary \ref{Thm1_20201215} are proved in Section \ref{Sec2_integral_F}, which actually brings a full discussion on effective bounds for each $b \geq 1$ and $\ell >0$. This section is of utmost importance for us, as it brings the foundations on the extremal problems in Fourier analysis that are connected to bounding the integral of $F(\alpha)$, and how one can properly explore them. For instance, the proof of the lower bound in \eqref{20201215_00:21}, which treads strikingly close to the conjectured value of $\ell+o(1)$ for large $\ell$, relies partly on the insight that Dirichlet kernels cannot be large and negative. In fact, letting 
$$c_0 := \min_{x\in\mathbb{R}} \frac{\sin x}{x} = -0.21723\ldots\,,$$
we see how the number 
\begin{equation}\label{20210107_14:18}
1 + \frac{c_0}{3} = 0.92758\ldots
\end{equation}
appears naturally in our discussion. We first obtain \eqref{20201215_00:21} with any constant smaller than \eqref{20210107_14:18} multiplying $\ell$ in the lower bound, and any constant greater than $4/3$ multiplying $\ell$ in the upper bound. A minor, yet conceptually important, improvement leads us to sharpen these multiplying factors to $0.9278$ in the lower bound and to $1.3302$ in the upper bound. Our general theoretical framework may be amenable to further slight numerical refinements through the search of more complicated test functions. A posteriori, the reader will notice that the fundamental pillar of the Section  \ref{Sec2_integral_F} is Theorem \ref{Prop_20201217_09:51}, a powerful general result that governs all the others in the section, including Theorem \ref{Thm1_20201215*} and Corollary \ref{Thm1_20201215}. We need a little bit of preparation in order to present it.

\subsubsection{Primes in short intervals } In \eqref{20190715_05:44} and \eqref{20190705_11:49am} below, we properly define the precise constants ${\bf L}^\pm$ which can be approximated by
\begin{align*}
{ \bf L^-} = 0.9028\ldots \ \ \ {\rm and} \ \ \ \ { \bf L^+} = 1.0736\ldots.
\end{align*}
Using the definitions of ${\bf L}^{\pm}$, a Tauberian argument, and the estimates for the integral of $F(\alpha)$ in bounded intervals, we deduce upper and lower bounds for the (weighted) variance of primes in short intervals.
\begin{theorem}\label{Thm_20201215_01:27*}
Assume RH and let $\varepsilon >0$ be an arbitrary number. For large $\beta$, as $T \to \infty$, we have
\begin{equation*}%\label{20201208_18:04*}
 \Big(\big({\bf L}^- {\bf C}^- - \varepsilon\big)\beta + o(1)\Big)\frac{\log^2T}{T} \le J(\beta,T) \le \Big( \big({\bf L}^+ {\bf C}^+ + \varepsilon\big)\beta + o(1)\Big)  \frac{\log^2T}{T}.
\end{equation*}
\end{theorem}

The multiplying factors ${\bf L}^{\pm}$ arise from what we call {\it sunrise approximations}\, for the Fej\'er kernel. Using the bounds for ${\bf C}^\pm$ in \eqref{20210205_11:56}, we deduce the following result. % are led to the following corollary. 

\begin{corollary}\label{Thm_20201215_01:27}
Assume RH. For large $\beta$, as $T \to \infty$, we have
\begin{equation}\label{20201208_18:04}
\big( 0.8376 \, \beta +o(1) \big) \frac{\log^2T}{T} \le J(\beta,T) \le \big( 1.4283 \, \beta+o(1) \big) \frac{\log^2T}{T}.
\end{equation}
\end{corollary}

Previously, the best known bounds in \eqref{20201208_18:04} were implicit in the work of Goldston and Gonek \cite{GG}, yielding $0.153$ in place of $0.8376$ in the lower bound, and $10.824$ in place of $1.4283$ in the upper bound. In Section \ref{Sec_PSI_new}, we present a full discussion on bounds for $J(\beta,T)$ for each $\beta >1$.

\smallskip

\subsubsection{The second moment of the logarithmic derivative of $\zeta(s)$} Our next result establishes the sharpest known bounds for $I(a,T)$, for any fixed $a>0$, assuming RH. Our upper bound for $I(a,T)$ uses a formula of Goldston, Gonek, and Montgomery \cite[Theorem 1]{GGM} combined with the solution of the Beurling--Selberg extremal problem for the Poisson kernel given in \cite{CChiM, CLV}. This argument is inspired by the previous calculations in \cite{ChSo} and \cite{CCM}, where explicit formula methods were combined with the solutions of the Beurling--Selberg extremal problem to give the sharpest known bounds for the modulus and argument of $\zeta(s)$ on the critical line, assuming RH. Our lower bound for $I(a,T)$ also uses \cite[Theorem 1]{GGM} together with a method developed in \cite[Theorem 7]{CCLM} to prove the existence of small gaps between the non-trivial zeros of $\zeta(s)$ using known pair correlation estimates.

\begin{theorem} \label{log zeta}
Assume RH. Then, for $T^{-1}(\log T)^{5/2}\leq a \le (\log T)^{1/4}/(\log\log T)^{1/2}$, we have
\[
\big(1+o(1)\big) \, U^-(a) \,T\log^2 T  \le I(a,T) \leq \big(1+o(1)\big) \, U^+(a) \,T\log^2 T 
\]
as $T \to \infty$, where 
\[
\begin{split}
U^-(a)& = \frac{1-(1+2 a) \, e^{-2 a}}{4 a^{2}} + \left(\frac{1}{2 a}+\frac{1}{\sqrt{3}}\right) e^{-2 a(1+1 / \sqrt{3})},
\\
U^+(a)&= \dfrac{\coth a}{4a^2}-\dfrac{(\csch\, a)^2}{4a} + \dfrac{\coth a}{2}-\dfrac{1}{2},
\end{split}
\]
and the terms of $o(1)$ are $O\big(1/\sqrt{\log\log T}\big)$.
\end{theorem}

\smallskip

To compare Theorem \ref{log zeta} to the conjectural asymptotic formula in (IV), let
\[
G^\pm(a) = U^\pm(a) \bigg/ \left(\frac{1-e^{-2a}}{4 a^2}\right).
\]
We then have $G^-(0^+)=1$, $G^+(0^+)=4/3$, $\min_{a>0} G^-(a) = 0.899\ldots$ attained at $a_0=0.998\ldots$, and $\max_{a>0} G^+(a) = 1.434\ldots$ attained at $a_0=0.620\ldots$ . Both $G^\pm(a) \to 1$ rapidly as $a \to \infty$, for example $G^-(a) \ge 0.999$ if $a \ge 4.55$ and $G^+(a) \le 1.001$ if $a \ge 5.83$. See Figure \ref{figure1}. Assuming RH, in the range $T^{-1}\log^3 T \leq a \ll 1$, Goldston, Gonek, and Montgomery \cite{GGM} had previously proved that 
\[
\big(1+o(1)\big) \, V^-(a) \,T\log^2 T  \le I(a,T) \leq \big(1+o(1)\big) \, V^+(a) \,T\log^2 T,
\]
where
\[
V^-(a)=\frac{1\!-\!(1\!+\!2 a) \, e^{-2 a}}{4 a^{2}}+\frac{2}{3\,(e^{6 a}\!-\!e^{2 a})} \ \text{ and } \ V^+(a)=\frac{1\!-\!(1\!+\!2 a) \, e^{-2 a}}{4 a^{2}}+\frac{29}{12\,(e^{2 a}\!-\!1)}.
\]
The bounds in Theorem \ref{log zeta} are sharper for any fixed $a>0$ %\footnote{As $a\to\infty$, Theorem \ref{log zeta} gives an upper bound of $\left(\frac{1}{4a^2}+e^{-2a} +O\Big(\frac{e^{-2a}}{a}\Big) \right)T\log^2T$ while GGM gives $\left(\frac{1}{4a^2}+\frac{29}{12}e^{-2a} +O\Big(\frac{e^{-2a}}{a}\Big) \right)T\log^2T$ } 
and substantially better for small $a$. See Figure \ref{figure2}. 

\subsubsection{Hilbert spaces and the pair correlation of zeta zeros} In Appendix B, we revisit the framework of \cite{CCLM} to find the sharp form of an embedding between two Hilbert spaces of entire functions naturally connected to Montgomery's pair correlation conjecture. Using tools from complex analysis, interpolation, and variational methods, we are led to the intriguing result presented in Theorem \ref{Thm_20201218_14:43}.

\begin{figure} 
\includegraphics[scale=.4]{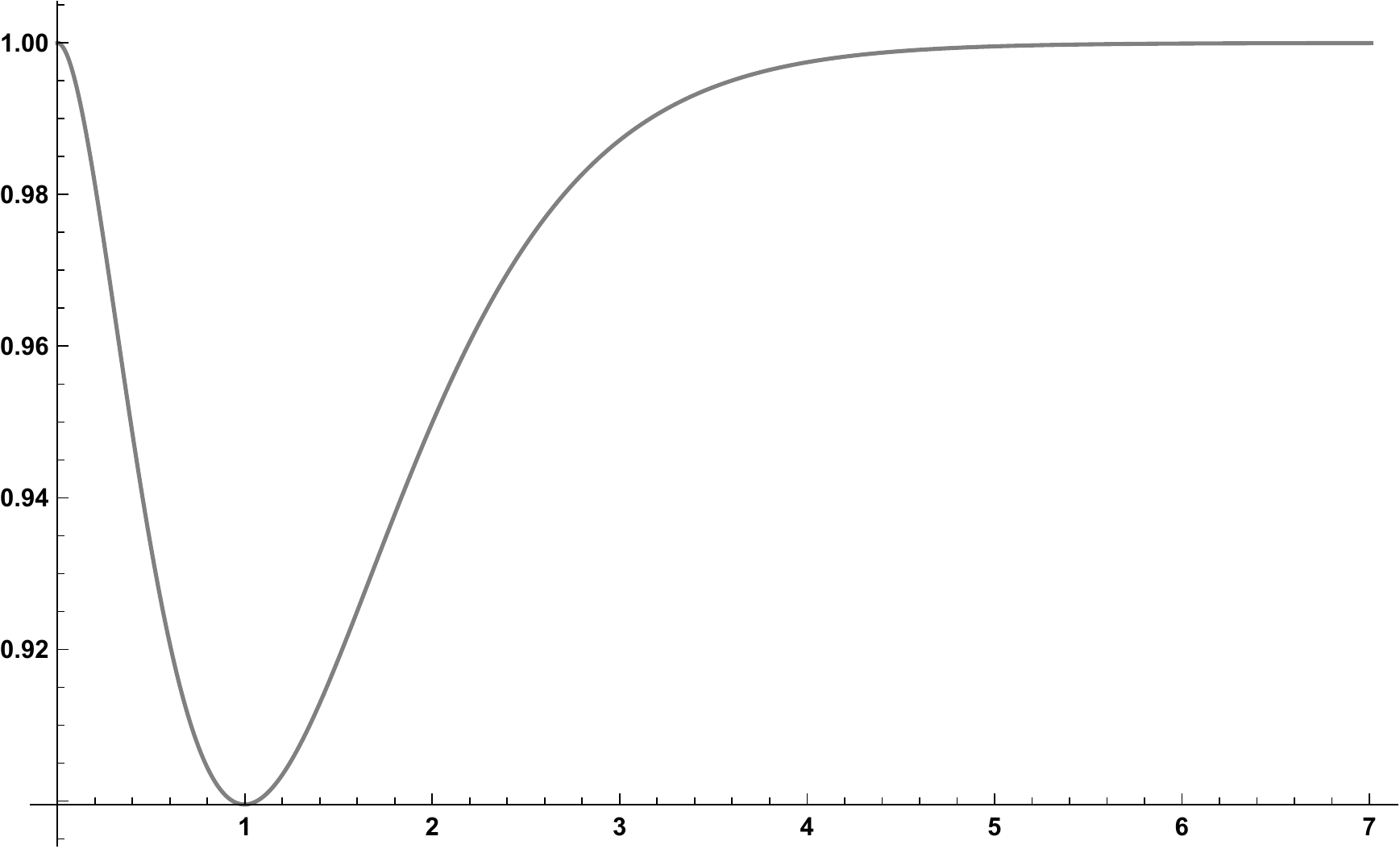} \qquad 
\includegraphics[scale=.4]{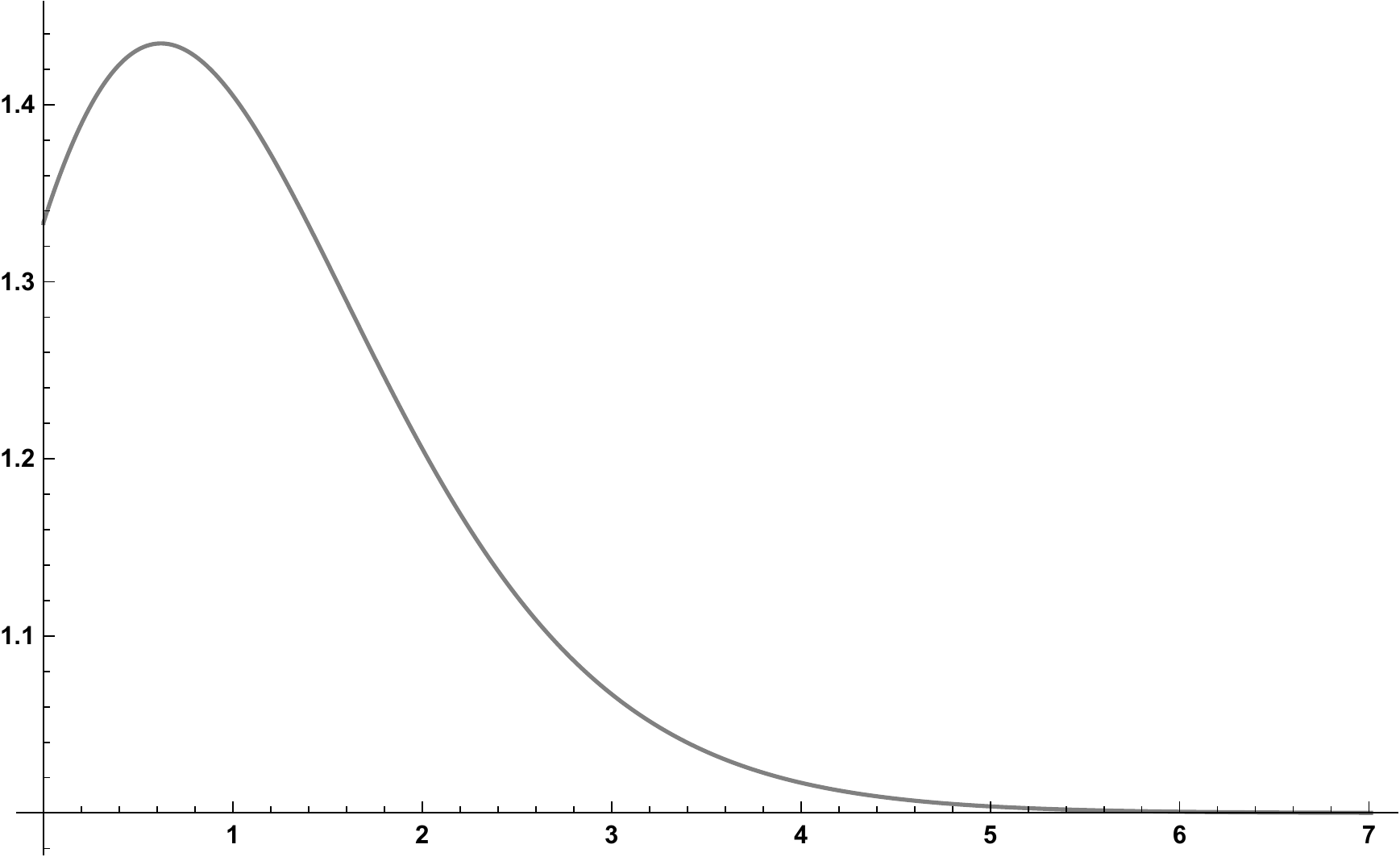} 
\caption{Plots of $G^-(a)$ and $G^+(a)$ for $0\le a\le 7$. }
\label{figure1}
\end{figure}

\begin{figure} 
	\includegraphics[scale=.4]{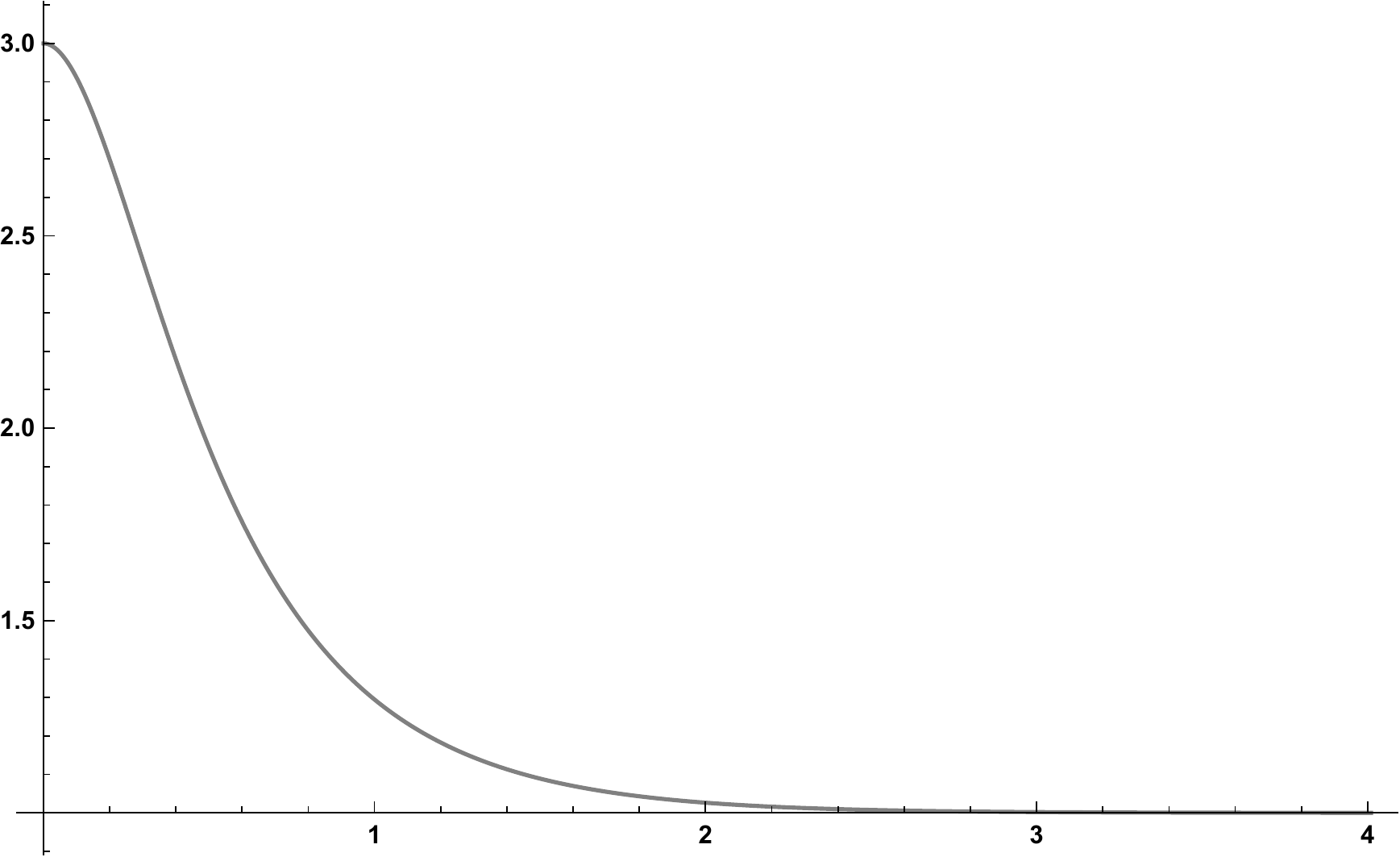} \qquad 
	\includegraphics[scale=.4]{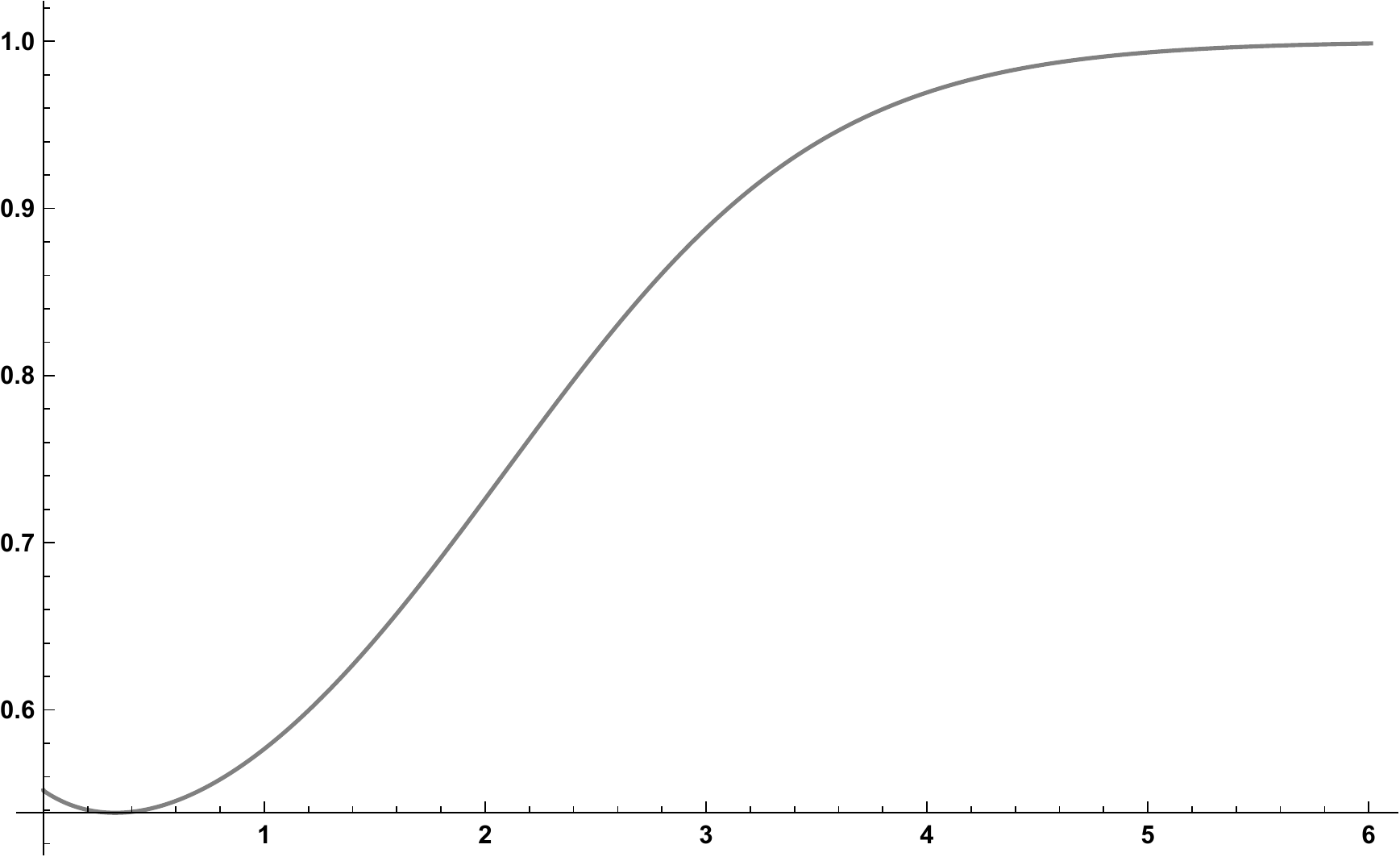} 
	\caption{Plots of $U^-(a)/V^-(a)$ for $0\le a\le 4$ and $U^+(a)/V^+(a)$ for $0\le a\le 6$. }
	\label{figure2}
\end{figure}

\subsection{Notation} Throughout the paper, $\lfloor x \rfloor$ denotes the largest integer that is less than or equal to $x$; $\lceil x \rceil$ denotes the smallest integer that is greater than or equal to $x$; and $\{x\} = x - \lfloor x \rfloor$ denotes the fractional part of $x$. We also write $x_+ := \max\{x,0\}$ and $\chi_{E}$ for the characteristic function of a set $E$. The real part of complex number $z$ is denoted by ${\rm Re}(z)$ and its imaginary part by ${\rm Im}(z)$.

\pagebreak 

\section{The integral of $F(\alpha)$ in bounded intervals}\label{Sec2_integral_F}

\subsection{Fourier optimization} \label{Fourier_Opt}We start with a broad principle to generate upper and lower bounds for the integral of $F(\alpha)$ in bounded intervals. This is motivated by some particular constructions %already present in the works of Goldston \cite{G2} and Goldston and Gonek \cite{GG}, and we now consider matters in a more general framework.
of Goldston \cite{G2} and Goldston and Gonek \cite{GG}, though we now set up the problem in a more general framework.

\smallskip

Throughout the paper we let $\mathcal{A}$ be the class of continuous, even, and non-negative functions $g \in L^1(\R)$ such that $\widehat{g}(\alpha) \leq 0$ for $|\alpha| \geq 1$. One can check, via approximations of the identity, that if $g \in \mathcal{A}$ then $\widehat{g} \in L^1(\R)$. For each $g \in \mathcal{A}$, we define the quantity 
\begin{equation}\label{20201217_11:41}
\rho(g) := \widehat{g}(0) +  \int_{-1}^1 \widehat{g}(\alpha)\,|\alpha| \, \mathrm{d}\alpha\,,
\end{equation}
which is always non-negative since $|\widehat{g}(\alpha)| \leq \widehat{g}(0)$ for all $\alpha \in \R$. In fact, \eqref{20201217_11:41} is strictly positive if $g \neq 0$. If $g \in \mathcal{A}$, from \eqref{inversion}, the fact that $F$ is non-negative, and \eqref{F formula}, we observe that
\begin{align}\label{20201216_11:09}
\begin{split}
  \frac{2\pi}{T \log T} \sum_{0<\gamma,\gamma'\le T} & g\!\left((\gamma-\gamma') \frac{\log T}{2\pi} \right)  w(\gamma-\gamma')  = \ \int_{-\infty}^{\infty}\widehat{g}(\alpha) \, F(\alpha,T) \, \mathrm{d}\alpha  \\
  & \leq \int_{-1}^{1}\widehat{g}(\alpha) \, F(\alpha,T) \, \mathrm{d}\alpha = \rho(g) + o(1) 
 \end{split}
\end{align}
as $T \to \infty$. We define $\mathcal{A}_0 \subset \mathcal{A}$ as the subclass of continuous, even, and non-negative functions $g \in L^1(\R)$ such that ${\rm supp}(\widehat{g}) \subset [-1,1]$. If $g \in \mathcal{A}_0$, then we have equality in \eqref{20201216_11:09}, and also the alternative representation
\begin{equation}\label{20210205_14:19}
\rho(g) = g(0) + \int_{-\infty}^{\infty} g(x) \left\{ 1 - \left( \frac{\sin \pi x}{\pi x}\right)^2
\right\} \dx\,,
\end{equation} 
which follows by Plancherel's theorem.

\subsubsection{Three extremal problems in Fourier analysis} We now introduce the following problems.

\subsubsection*{Extremal problem 1 {\rm (EP1)}} Let $\ell >0$. Consider a finite collection of functions $g_1, g_2, \ldots, g_N \in \mathcal{A}$ and points $\xi_1, \xi_2, \ldots, \xi_N \in \R$ such that 
\begin{equation}\label{20201216_11:47}
\sum_{j =1}^N \widehat{g_j} (\alpha - \xi_j) \geq \chi_{[0,\ell]}(\alpha)
\end{equation}
for all $\alpha \in \R$. Over all such possibilities, find the infimum
\begin{equation}\label{20210113_08:57}
\mathcal{W}^+(\ell) := \inf \sum_{j=1}^{N} \rho(g_j).
\end{equation}

\subsubsection*{Extremal problem 2 {\rm (EP2)}} Let $\ell >0$. Consider a finite collection of functions $g_1, g_2, \ldots, g_N \in \mathcal{A}$ and points $\xi_1, \xi_2, \ldots, \xi_N \in \R$ such that 
\begin{equation}\label{20201216_12:13}
\sum_{j =1}^N \widehat{g_j} (\alpha - \xi_j) \leq \chi_{[0,\ell]}(\alpha)
\end{equation}
for all $\alpha \in \R$. Over all such possibilities, find the supremum
\begin{equation}\label{20201216_12:14}
\mathcal{W}^{-}(\ell) := \sup \sum_{j=1}^{N} \big(2g_j(0) - \rho(g_j)\big).
\end{equation}

\subsubsection*{Extremal problem 3 {\rm (EP3)}} Let $b, \beta \in \R$ with $b < \beta$. Consider a finite collection of functions $g_1, g_2, \ldots, g_N \in \mathcal{A}$, points $\eta_1, \eta_2, \ldots, \eta_N \in \R$, and values $\frak{r}_1, \frak{r}_2, \ldots, \frak{r}_N \in (-\infty,1]$ with $\frak{r}_j \leq 0$ if $g_j \in \mathcal{A}\setminus \mathcal{A}_0\ (j =1,2, \ldots, N)$, such that 
\begin{equation}\label{20210107_21:13}
\sum_{j =1}^N \widehat{g_j} (\alpha - \eta_j) \leq \chi_{[b,\beta]}(\alpha)
\end{equation}
for all $\alpha \in \R$, and 
\begin{equation}\label{20210108_10:11}
{\rm Re} \left( \sum_{j=1}^N e^{2\pi i \eta_j x} g_j(x)\right) \geq \sum_{j=1}^N \frak{r}_j \,g_j(x)
\end{equation}
for all $x \in \R$. Over all such possibilities, find the supremum
\begin{equation}\label{20210111_09:03}
\mathcal{W}_*^{-}(b, \beta) := \sup \sum_{j=1}^{N} \big(g_j(0)  + \frak{r}_j \big(\rho(g_j) - g_j(0)\big)\big).
\end{equation} 

\noindent{\sc Remark 1:} Note that by a uniform translation of all the $\xi_j$'s one can consider any interval of length $\ell$ in \eqref{20201216_11:47} and \eqref{20201216_12:13} instead of the interval $[0,\ell]$. The situation is slightly different in (EP3) since, for fixed $g_j$'s and $\frak{r}_j$'s, condition \eqref{20210108_10:11} is not necessarily invariant under translations of the $\eta_j$'s, and hence the answer may depend on the particular interval $[b, \beta]$ that we choose in \eqref{20210107_21:13}. Throughout this section, we reserve the variable $\ell$ for the length of the interval, hence the change of variables $\beta = b + \ell$ is sometimes used. In \eqref{20210108_10:11} note that 
the choice $\frak{r}_1 = \frak{r}_2 = \ldots = \frak{r}_N = -1$ is always admissible.

\smallskip

\noindent {\sc Remark 2:} In the next subsections, we see that collections of functions and points that satisfy \eqref{20201216_11:47}, \eqref{20201216_12:13}, or \eqref{20210107_21:13}--\eqref{20210108_10:11}  indeed exist. We do not take the supremum and infimum over empty sets.

\smallskip
 
At this point we collect some basic facts about the newly introduced functions $\mathcal{W}^{+}, \mathcal{W}^{-}$ and $\mathcal{W}^{-}_*$. 
\begin{proposition}\label{Prop_basic_prop}
The following statements hold:
\begin{itemize}
\item[(i)] The functions $\ell \mapsto \mathcal{W}^+(\ell)$, $\ell \mapsto \mathcal{W}^-(\ell)$ and $\ell \mapsto \mathcal{W}^-_*(b, b+\ell)$ are non-decreasing for $b \in \mathbb R$ and $\ell > 0.$
\smallskip
\item[(ii)] For each $b \in \R$ and $\ell >0$ we have
\begin{equation}\label{20210108_10:15}
\mathcal{W}^-(\ell) \leq \mathcal{W}^-_*(b , b + \ell).
\end{equation}
\item[(iii)] For each $\ell_1, \ell_2 >0$ we have
\begin{equation}\label{20210108_10:34}
\mathcal{W}^+(\ell_1 + \ell_2) \leq \mathcal{W}^+(\ell_1) + \mathcal{W}^+(\ell_2) \ \ \ {\rm and} \ \ \ \mathcal{W}^-(\ell_1 + \ell_2) \geq \mathcal{W}^-(\ell_1) + \mathcal{W}^-(\ell_2).
\end{equation}
\item[(iv)] For $b < c < d$ we have
\begin{equation}\label{20210108_11:14}
\mathcal{W}^-_*(b, d) \geq \mathcal{W}^-_*(b, c) + \mathcal{W}^-_*(c, d).
\end{equation}
\end{itemize}
\end{proposition}
\begin{proof}
(i) This should be clear from the definitions of the extremal problems (EP1), (EP2) and (EP3). 

\smallskip

\noindent (ii) Assume that \eqref{20201216_12:13} is verified. Then, letting $\eta_j = \xi_j + b$, we verify \eqref{20210107_21:13} with $\beta = b + \ell$. We may choose $\frak{r}_1 = \frak{r}_2 = \ldots = \frak{r}_N = -1$ in \eqref{20210108_10:11} to arrive at inequality \eqref{20210108_10:15}.

\smallskip

\noindent (iii) If $\big( \{g_{1,j}\}_{j=1}^{N_1}, \{\xi_{1,j}\}_{j=1}^{N_1}\big)$ verifies \eqref{20201216_11:47} with $\ell = \ell_1$ and $\big( \{g_{2,j}\}_{j=1}^{N_2}, \{\xi_{2,j}\}_{j=1}^{N_2}\big)$ verifies \eqref{20201216_11:47} with $\ell = \ell_2$, then the collection $\big( \{g_{3,j}\}_{j=1}^{N_1 + N_2}, \{\xi_{3,j}\}_{j=1}^{N_1 + N_2}\big)$ verifies \eqref{20201216_11:47} with $\ell = \ell _1 + \ell_2$, where 
\begin{equation*}
g_{3,j} = \left\{ 
\begin{array}{ll}
g_{1,j},&\ {\rm for} \ 1 \leq j \leq N_1;\\
g_{2,j-N_1}&\ {\rm for} \ N_1 +1 \leq j  \leq N_1 + N_2;
\end{array}
\right.
\ \ ; \ \ 
\xi_{3,j} = \left\{ 
\begin{array}{ll}
\xi_{1,j},&\ {\rm for} \ 1 \leq j \leq N_1;\\
\xi_{2,j-N_1} + \ell_1&\ {\rm for} \ N_1 +1 \leq j  \leq N_1 + N_2.
\end{array}
\right.
\end{equation*}
This leads us to \eqref{20210108_10:34} for $\mathcal{W}^+$. A similar concatenation argument yields the inequality for $\mathcal{W}^-$.

\smallskip

\noindent (iv) Assume that the configuration $\big( \{g_{1,j}\}_{j=1}^{N_1}, \{\eta_{1,j}\}_{j=1}^{N_1}, \{\frak{r}_{1,j}\}_{j=1}^{N_1}\big)$ verifies \eqref{20210107_21:13} -- \eqref{20210108_10:11} for the interval $[b,c]$, and that $\big( \{g_{2,j}\}_{j=1}^{N_2}, \{\eta_{2,j}\}_{j=1}^{N_2},  \{\frak{r}_{2,j}\}_{j=1}^{N_2}\big)$ verifies \eqref{20210107_21:13}--\eqref{20210108_10:11} for the interval $[c,d]$. Then the collection $\big( \{g_{3,j}\}_{j=1}^{N_1}, \{\eta_{3,j}\}_{j=1}^{N_1 + N_2},  \{\frak{r}_{3,j}\}_{j=1}^{N_1 + N_2}\big)$ verifies \eqref{20210107_21:13}--\eqref{20210108_10:11} for the interval $[b,d]$, where
\begin{equation*}
g_{3,j} = \left\{ 
\begin{array}{ll}
g_{1,j},&\ {\rm for} \ 1 \leq j \leq N_1;\\
g_{2,j-N_1}&\ {\rm for} \ N_1 +1 \leq j  \leq N_1 + N_2;
\end{array}
\right.
\ \ ; \ \ 
\eta_{3,j} = \left\{ 
\begin{array}{ll}
\eta_{1,j},&\ {\rm for} \ 1 \leq j \leq N_1;\\
\eta_{2,j-N_1} &\ {\rm for} \ N_1 +1 \leq j  \leq N_1 + N_2;
\end{array}
\right.
\end{equation*}
and
\begin{equation*}
\frak{r}_{3,j} = \left\{ 
\begin{array}{ll}
\frak{r}_{1,j},&\ {\rm for} \ 1 \leq j \leq N_1;\\
\frak{r}_{2,j-N_1}&\ {\rm for} \ N_1 +1 \leq j  \leq N_1 + N_2.
\end{array}
\right.
\end{equation*}
This leads us to \eqref{20210108_11:14}.
\end{proof}

\subsubsection{A general bound} We now relate the three extremal problems introduced above to the integral of $F(\alpha)$ in the following general result.
\begin{theorem}\label{Prop_20201217_09:51}
Assume RH, let $b \in \R$ and $\ell >0$. Then, as $T\to\infty$, we have
\begin{align}\label{20210108_11:25}
	\!\!\!\!\!\!\! \mathcal{W}^-(\ell) +o(1) \leq \mathcal{W}^-_*(b, b+\ell)  + o(1)  \le \int_b^{b + \ell} F(\alpha,T) \, \d\alpha \le \mathcal{W}^+(\ell) +o(1).
\end{align}	
\end{theorem}
\begin{proof} The first inequality on the left-hand side of \eqref{20210108_11:25} was already established in  Proposition \ref{Prop_basic_prop} (ii).

\smallskip

Assume that \eqref{20201216_11:47} holds. Then, using \eqref{20201216_11:47}, \eqref{inversion} and \eqref{20201216_11:09} we have 
\begin{align*}
\int_b^{b + \ell} F(\alpha) \, \d\alpha &\le  \sum_{j=1}^N \int_{\mathbb{R}} F(\alpha) \, \widehat{g_j}(\alpha-\,b - \xi_j) \, \d\alpha
\\
&= \frac{2\pi}{T\log T} \sum_{j=1}^N \ \sum_{0<\gamma,\gamma'\le T} T^{i(b+\xi_j)(\gamma-\gamma')} \, g_j\!\left((\gamma-\gamma')\frac{\log T}{2\pi}\right) \,w(\gamma-\gamma')  
\\
&\le \frac{2\pi}{T\log T} \sum_{j=1}^N \ \sum_{0<\gamma,\gamma'\le T} g_j\!\left((\gamma-\gamma')\frac{\log T}{2\pi}\right) \,w(\gamma-\gamma')  
\\
& \leq \sum_{j=1}^{N} \rho(g_j) + o(1),
\end{align*}
which leads us to the upper bound in \eqref{20210108_11:25}.

\smallskip

Now assume that \eqref{20210107_21:13} and \eqref{20210108_10:11} hold, with $\beta = b + \ell$. For the lower bound, we are inspired by a trick of Goldston \cite[p.~172]{G2}. Letting $m_\gamma$ denote the multiplicity of a zero $\frac{1}{2}+i\gamma$ of $\zeta(s)$, we use \eqref{20210107_21:13}, \eqref{inversion}, \eqref{20210108_10:11}, and \eqref{20201216_11:09} (recall that $\frak{r}_j \leq 0$ if $g_j \in \mathcal{A}\setminus \mathcal{A}_0$) to get
\begin{align}
\int_b^{b + \ell} F(\alpha) \, \d\alpha &\ge  \sum_{j=1}^{N} \int_{\mathbb{R}} F(\alpha) \, \widehat{g_j}(\alpha-\eta_j) \, \d\alpha \nonumber 
\\
&= \frac{2\pi}{T\log T} \sum_{j=1}^N \ \sum_{0<\gamma,\gamma'\le T} T^{\,i\,\eta_j(\gamma-\gamma')}\, g_j\left((\gamma-\gamma')\frac{\log T}{2\pi}\right) \,w(\gamma-\gamma')   \nonumber 
\\
& =  \frac{2\pi}{T\log T}   \sum_{j=1}^{N}\left\{ g_j(0) \sum_{0<\gamma \le T} m_\gamma \, +  \sum_{\substack{0<\gamma,\gamma'\le T \\ \gamma \ne \gamma'}} T^{\,i\,\eta_j(\gamma-\gamma')}\, g_j\left((\gamma-\gamma')\frac{\log T}{2\pi}\right) \,w(\gamma-\gamma')  \right\} \nonumber 
\\
& \geq   \frac{2\pi}{T\log T}   \sum_{j=1}^{N}\left\{ g_j(0) \sum_{0<\gamma \le T} m_\gamma \, + \frak{r}_j \sum_{\substack{0<\gamma,\gamma'\le T \\ \gamma \ne \gamma'}} g_j\left((\gamma-\gamma')\frac{\log T}{2\pi}\right) \,w(\gamma-\gamma')  \right\}    \label{20210111_09:32}\\
& = \frac{2\pi}{T\log T}   \sum_{j=1}^{N}\left\{  g_j(0)\,(1 - \frak{r}_j) \sum_{0<\gamma \le T} m_\gamma \, + \frak{r}_j \sum_{0<\gamma,\gamma'\le T}  g_j\left((\gamma-\gamma')\frac{\log T}{2\pi}\right) \,w(\gamma-\gamma')  \right\} \nonumber  \\
& \ge    \sum_{j=1}^{N} \big(g_j(0)  + \frak{r}_j \big(\rho(g_j) - g_j(0)\big)\big) + o(1).\nonumber 
\end{align}
Here we have used the trivial bound
\begin{equation*}
\sum_{0<\gamma\le T} m_\gamma \ge \sum_{0<\gamma\le T} 1 \sim \frac{T\log T}{2\pi}, \quad \text{as } T \to \infty,
\end{equation*}
to derive the final inequality. This leads us to the lower bound for the integral of $F(\alpha)$ in \eqref{20210108_11:25}.
\end{proof}

\smallskip

\noindent{\sc Remark:} It is an interesting problem to determine when the lower bounds in Theorem \ref{Prop_20201217_09:51} start beating the trivial bound of $0$. For instance, in Theorem \ref{Prop_20201220_16:37} below we show that $\mathcal{W}^-(\ell) >0$ for $\ell > 6 - 2\sqrt{6} = 1.10102\ldots$

\smallskip

In the case $b=1$ we may take advantage of the symmetry around the origin and \eqref{F formula} to provide alternative upper and lower bounds as follows.

\begin{corollary}\label{Cor_20210108}
Assume RH and let $\beta > 1$. Then, as $T\to\infty$, we have
\begin{equation}\label{20210108_13:52}
	\frac{\mathcal{W}^-_*(-\beta, \beta)}{2} - 1  + o(1) \le \int_1^{\beta} F(\alpha,T) \, \d\alpha \le \frac{\mathcal{W}^+( 2\beta)}{2} - 1 +o(1).
\end{equation}	
\end{corollary}
\begin{proof}
The estimate in \eqref{F formula} implies that 
\begin{equation}\label{20201223_13:38}
\int_{-1}^1  F(\alpha) \, \mathrm{d} \alpha  =2+ o(1).
\end{equation}
Using \eqref{20201223_13:38} and the fact that $F(\alpha)$ is even we have
\begin{align}\label{20201223_13:39}
\int_{-\beta}^{\beta}  F(\alpha) \, \mathrm{d} \alpha  = 2 \int_1^{\beta} F(\alpha) \, \d\alpha + 2 + o(1).
\end{align}
The desired bounds in \eqref{20210108_13:52} now follow from \eqref{20201223_13:39} and Theorem \ref{Prop_20201217_09:51}.
\end{proof}
%\noindent{\sc Remark:} The fact that the bounds in \eqref{20210108_13:52} are in fact stronger than the bounds in \eqref{20210108_11:25} is a consequence of Proposition \ref{Prop_basic_prop} (iii) and (iv), together with the information that $\mathcal{W}^-_*(-1, 1) \leq 2 \leq \mathcal{W}^+( 2)$. The latter follows from \eqref{20201223_13:38} and Theorem \ref{Prop_20201217_09:51}.

\subsubsection{Strengths and limitations} Finding the exact answer in the general case of extremal problems (EP1), (EP2) and (EP3) above is, in principle, something non-trivial. There are too many parameters in play. On the other hand, an advantage of this method and Theorem \ref{Prop_20201217_09:51} is that, for a fixed interval $[b,b + \ell]$, it is possible to bring in sophisticated computational tools to approximate the solutions of these extremal problems.

\smallskip

As noted in Proposition \ref{Prop_basic_prop} (ii) and Theorem \ref{Prop_20201217_09:51}, the extremal problem (EP2) provides a weaker lower bound than (EP3), but has the advantage of being a simpler problem. In fact, if one wants to obtain effective estimates for all intervals in a more systematic way, it is simpler to narrow down the search to certain families of functions within the subclass $\mathcal{A}_0$ and work with (EP1) and (EP2) to start. We proceed along these lines in the next subsection. We note that the larger class $\mathcal{A}$ has proved useful to sharpen some bounds in the theory of the Riemann zeta-function via sophisticated numerical experimentation \cite{CGL} and, though numerics is not our main focus here, we have already laid the foundational theoretical framework for such endeavors.

\smallskip

Montgomery and Taylor \cite{M2} showed that for each function $0\neq g \in \mathcal{A}_0$ one has
\begin{equation}\label{20201217_11:08}
\frac{\rho(g)}{g(0)}\,  \geq \, {\bf C_{MT}}:= \frac{1}{2} + 2^{-\frac12} \cot\left(2^{-\frac12}\right) = 1.32749\ldots,
\end{equation}
with equality if and only if
$$g(x) = \frac{c}{(1- 2\pi^2 x^2)^2}\left( \cos (\pi x) - 2^{\frac12}\pi x \cot \big(2^{-\frac12}\big) \sin(\pi x) \right)^2 \ \ \ \ (c >0).$$
For an alternative proof using reproducing kernel Hilbert spaces, see \cite[Corollary 14]{CCLM}. See also \cite[Appendix A]{ILS}. Assuming that \eqref{20201216_11:47} holds, we integrate to get
\begin{align}\label{20201217_11:09}
\sum_{j =1}^N g_j(0) = \int_{\R} \left(\sum_{j =1}^N \widehat{g_j} (\alpha - \xi_j)\right) \d \alpha \geq \int_{\R} \chi_{[0,\ell]}(\alpha) \, \d \alpha = \ell.
\end{align}
If all functions $g_j$ are in the subclass $\mathcal{A}_0$, from \eqref{20201217_11:08} and \eqref{20201217_11:09}, we see that
\begin{equation}\label{20210111_11:56}
\sum_{j =1}^N \rho(g_j) \geq {\bf C_{MT}}\sum_{j =1}^N g_j(0) \geq {\bf C_{MT}} \, \ell = (1.32749\ldots)\,\ell.
\end{equation}
Analogously, if in extremal problem (EP2) we restrict our attention to functions $g_j$ in the subclass $\mathcal{A}_0$, by integrating \eqref{20201216_12:13} and using  \eqref{20201217_11:08}, we get
\begin{equation}\label{20210108_17:12}
\sum_{j=1}^{N} \big(2g_j(0) - \rho(g_j)\big) \leq \big( 2 - {\bf C_{MT}} \big)\, \ell = (0.67250\ldots)\,\ell.
\end{equation}
These are universal limitations of this method when using the extremal problems (EP1) and (EP2) restricted to the subclass $\mathcal{A}_0$. For the lower bound, in the regime when $\ell$ is large, we see in \S \ref{SS_Triangles} that we can in fact get very close to the threshold \eqref{20210108_17:12} but, at the end, with the refined framework of \S \ref{SS_Dir} we  see that the extremal problem (EP3) yields a substantially better lower bound. For the upper bound, we show in \S \ref{SS_Triangles} and \S \ref{Sub_Pf_Thm1} that we can get very close to the threshold \eqref{20210111_11:56}.

\subsection{Stacking triangles} \label{SS_Triangles} A simple and effective way to use Theorem \ref{Prop_20201217_09:51}, with the lower bound given by (EP2), is by considering the functions $\widehat{g_j}$ being triangles. The linearity allows for a reasonable control over restrictions \eqref{20201216_11:47} and \eqref{20201216_12:13}. In fact, the key observation here is that the superposition (addition) of equally spaced triangular graphs morally results in a constant function. This idea is already hinted in the work of Goldston and Gonek \cite[Lemma]{GG}, and we further explore it here. For $0 < \Delta \leq 1$, consider the Fourier pair
\begin{equation}\label{20210125_09:36}
K_{\Delta}(x) = \Delta \left( \frac{\sin \pi \Delta x}{\pi \Delta x} \right)^{\!2}  %\left\{ \begin{array}{cl} \left( \frac{\sin \pi x}{\pi x} \right)^2, \quad & \ \textrm{if} \ x \neq 0, \\ 1, \ \ \ \ \  & \ \textrm{if} \ x = 0, \end{array} \right. \ \ \  \ \ \  
\quad \text{and} \quad \widehat{K_{\Delta}}(\xi) = \left(1-\frac{|\xi|}{\Delta}\right)_{\!+}.
\end{equation}
Note that the graph of $\widehat{K_{\Delta}}$ is a triangle with base $2 \Delta$ (centered at the origin) and height $1$. In this case, \eqref{20201217_11:41} yields
\begin{equation}\label{20210111_09:07}
\rho(K_{\Delta}) =  1 + \frac{\Delta^2}{3}.
\end{equation}
We establish the following effective bounds. 
\begin{theorem}[Triangle bounds]\label{Prop_20201220_16:37}
Assume RH, let $b \geq1$, and let $\ell >0$. Then, as $T\to\infty$, we have
\begin{equation*}
	\mathcal{C}^-_{\blacktriangle}(\ell) + o(1) \le \int_b^{b + \ell} F(\alpha,T) \, \d\alpha \le \mathcal{C}^+_{\blacktriangle}(\ell) +o(1),
\end{equation*}	
where
\begin{equation}\label{20210108_17:51}
\mathcal{C}^+_{\blacktriangle}(\ell)  = \left\{
\begin{array}{ll}
\frac{4}{3}(\ell +1) + \frac{\{\ell\}^3}{12}  - \frac{\{\ell\}}{3}  - \frac{1}{4}\left( 1 - \{\ell\} - \{\ell\}^2\right)_{\!+} \,,\ \ {\rm for} \ \ \ell \geq 1;\\
\min\left\{\frac{4}{3}(\ell +1) + \frac{\ell^3}{12}  - \frac{\ell}{3} - \frac{1}{4}\left( 1 - \ell - \ell^2\right)_{\!+}  \, \,;\,\, (1+c) \left(1 + \frac{\ell^2(1+c)^2}{12c^2}\right)\right\} , \\
\hskip 2in   {\rm for} \ \ 0 < \ell \leq 1 ,\, {\rm with} \ c =\max\left\{ 6^{-1/3}\, \ell^{2/3}\,,\, \frac{\ell}{2 - \ell}\right\};
\end{array}
\right.
\end{equation}
and
\begin{equation} \label{20201216_12:03}
\mathcal{C}^-_{\blacktriangle}(\ell)  = \left\{
\begin{array}{ll}
\frac{2}{3}(\ell -1) - \frac{2\{\ell\}}{3} + \left(\frac{1 + \{\ell\}}{2}\right)\left( \{\ell\}  - \frac{(1 + \{\ell\})^2}{12}\right)_{\!+} \,,\ \ {\rm for} \ \ \ell \geq 2;\\
\left( \ell - 1 -\frac{\ell^2}{12}\right)_{\!+}\ ,\,  {\rm for} \ \ 0 < \ell \leq 2.
\end{array}
\right.
\end{equation}
\end{theorem}
Before moving on to the proof of Theorem \ref{Prop_20201220_16:37}, let us make a few comments. The main point of this theorem is to bring in some relatively simple bounds, that can be explicitly stated for all $\ell$. Nevertheless, we pay attention to some important details that could be useful in other contexts. For instance, note that the functions $\ell \mapsto \mathcal{C}^{\pm}_{\blacktriangle}(\ell)$ are continuous and non-decreasing. Note also that our bound $\mathcal{C}^{+}_{\blacktriangle}(\ell)$ (which comes from a particular choice of functions in (EP1)) establishes that  
\begin{equation}\label{20210108_15:02}
\lim_{\ell \to 0^+} \mathcal{W}^+(\ell) = \lim_{\ell \to 0^+} \mathcal{C}^+_{\blacktriangle}(\ell) = 1.
\end{equation}
In fact, from \eqref{F formula} we get  $\int_{-\varepsilon}^{\varepsilon} F(\alpha)\, \mathrm{d} \alpha  \geq1+ o(1)$ for any fixed $\varepsilon >0$. Then, from Theorem \ref{Prop_20201217_09:51} we get
$$1 \leq  \mathcal{W}^+(\ell) \leq \mathcal{C}^{+}_{\blacktriangle}(\ell)$$
for all $\ell >0$, and we may pass the limit as $\ell \to 0^+$ to obtain \eqref{20210108_15:02}. Recall that we cannot rule out the existence of delta spikes in $F(\alpha)$ for $|\alpha| \geq 1$. The connection between this phenomenon and the so-called alternative hypothesis to Montgomery's strong pair correlation conjecture is investigated by Baluyot in \cite{Bal}. 

\smallskip

In the regime $0 < \ell \leq 1$, our upper bound $\mathcal{C}^+_{\blacktriangle}(\ell)$ is realized by the first function for $1/6 \leq \ell \leq \theta_1 = 0.3576\ldots$ and $ 0.7222\ldots = \theta_2 \leq \ell \leq 1$, and by the second function for $0 < \ell \leq 1/6$ and $\theta_1 \leq \ell \leq \theta_2$ (and in this range the transition of $c$ occurs at $\theta_3 = 0.5297\ldots$).
We note that the lower bound $\mathcal{C}^{-}_{\blacktriangle}(\ell)$ in \eqref{20201216_12:03} starts to be non-trivial at $\ell = 6 - 2\sqrt{6} = 1.10102\ldots$. Finally, we note that Theorem \ref{Prop_20201220_16:37} recovers a result of Goldston and Gonek \cite[Lemma, Eqs.(3), ~(4) and (5)]{GG} in the cases $0 \leq \ell \leq 2$ (lower bound) and $\ell =1$ (upper bound), and refines it in all the other cases. Figure \ref{figure_triangle_bounds} brings the plot of our triangle bounds for small values of $\ell$.

\begin{proof}[Proof of Theorem \ref{Prop_20201220_16:37}] The idea here is simply to establish that 
\begin{equation}\label{20210110_11:54}
\mathcal{C}^{-}_{\blacktriangle}(\ell) \leq \mathcal{W}^- (\ell) \leq \mathcal{W}^+(\ell) \leq \mathcal{C}^{+}_{\blacktriangle}(\ell)\,,
\end{equation}
and the result will follow from Theorem \ref{Prop_20201217_09:51}. Let us split the proof into its different regimes.

\medskip

\noindent {\it Step 1. Upper bound}. The strategy here is to consider $n$ big triangles and $2$ small triangles, one at each end, to adjust for the fractional part of $\ell$. Specifically, in the setup of extremal problem (EP1), we consider a configuration with $N = n+2$ functions given by $\widehat{g_2} = \widehat{g_3} = \ldots = \widehat{g_{n+1}} =  \widehat{K_{1}}$ (the triangle of height $1$ and base $2$; if $n=0$ this block is disregarded) and $\widehat{g_1} = \widehat{g_{n+2}} = \Delta \widehat{K_{\Delta}}$ (the triangle of height $\Delta$ and base $2\Delta$), where $0 < \Delta \leq 1$. Assume further that 
\begin{equation}\label{20201221_14:05}
(n-1) + 2 \Delta = \ell
\end{equation}
and observe that condition \eqref{20201216_11:47} is verified for the translates given by $\xi_1 = 0$; $\xi_j = (j-2) + \Delta$, for $j = 2,3,\ldots, n+1$; and $\xi_{n+2} =  (n-1) + 2\Delta $. For this particular configuration, we have 
\begin{equation}\label{20201221_14:08}
\sum_{j=1}^{n+2} \rho(g_j) = \frac{4n}{3} + 2\Delta \left(1 + \frac{\Delta^2}{3}\right).
\end{equation}
When $\ell \in \N$, since $0 < \Delta \leq 1$, identity \eqref{20201221_14:05} can only be verified if $(n, \Delta) = (\ell, \tfrac12)$ or $(\ell -1, 1)$. Among these two possibilities, the former optimizes \eqref{20201221_14:08}, yielding the upper bound $\tfrac43 \ell + \tfrac{13}{12}$. When $\ell \notin \N$, from \eqref{20201221_14:05} we may have $(n, \Delta) = \big(\lfloor \ell \rfloor +1, \{\ell\}/2\big)$ or $\big(\lfloor \ell \rfloor, (1+\{\ell\})/2\big)$. The minimum of these two in \eqref{20201221_14:08} yields the quantity:
$$ \frac{4}{3}(\ell +1) + \frac{\{\ell\}^3}{12} - \frac{\{\ell\}}{3} - \frac{1}{4}\left( 1 - \{\ell\} - \{\ell\}^2\right)_+$$
Note that the transition between the two possibilities occurs when $\{\ell\} = \frac{\sqrt{5} - 1}{2}$.

\begin{figure} 
\includegraphics[scale=0.6]{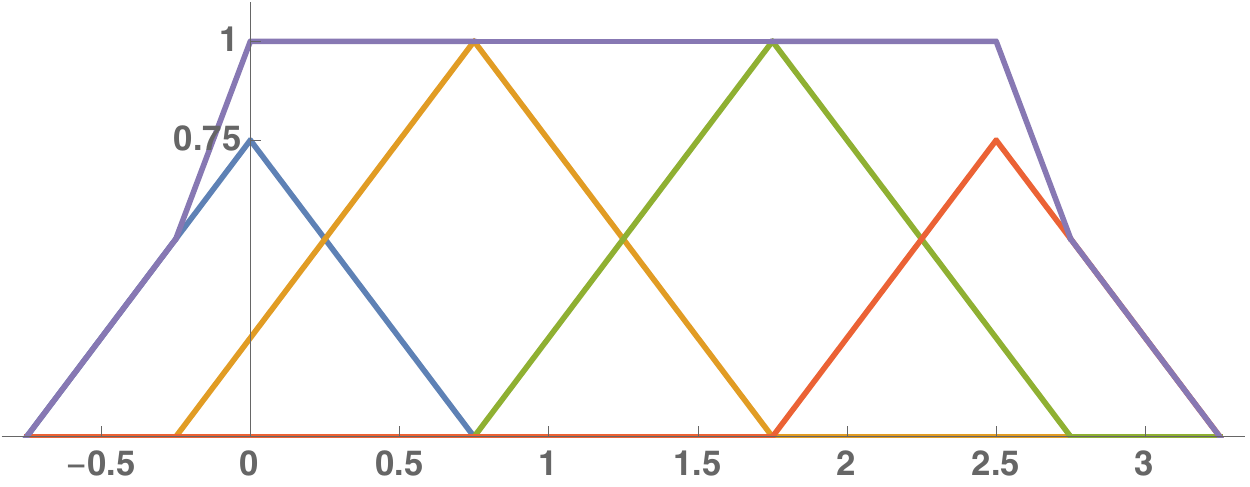}  \qquad 
\includegraphics[scale=0.6]{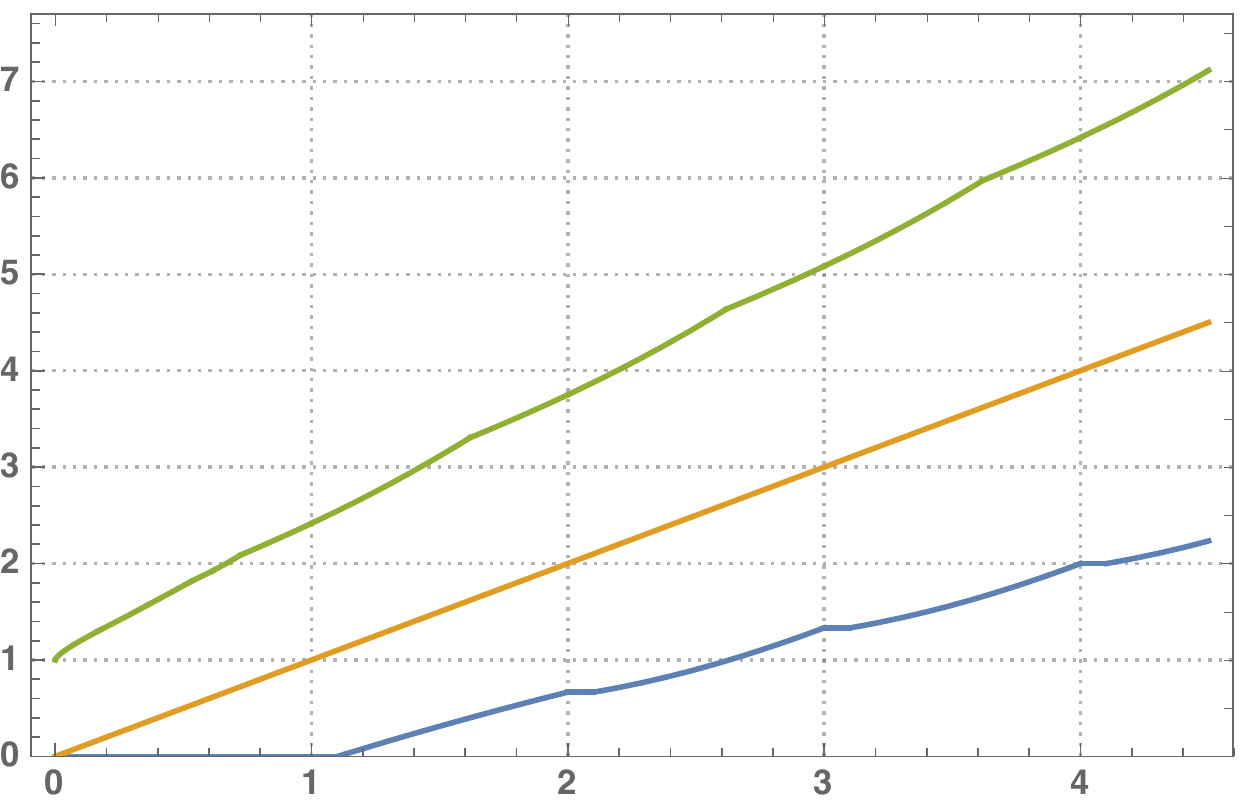}
\caption{On the left, the birth of the idea. This is the construction of the upper bound $\mathcal{C}^+_{\blacktriangle}(\ell)$ when $\ell = 2.5$, with $n=2$ and $\Delta = 3/4$, where the triangular graphs add up to the function on the top (in purple), that majorizes the characteristic function of the interval $[0,2.5]$. On the right, the plots of $\ell \mapsto \mathcal{C}^+_{\blacktriangle}(\ell)$ (in green), $\ell \mapsto \mathcal{C}^-_{\blacktriangle}(\ell)$ (in blue) and the conjectured asymptotic $\ell$ (in orange), for $0 \leq \ell 
	\leq 4.5$. }
	\label{figure_triangle_bounds}
\end{figure}

\medskip

\noindent {\it Step 2. Alternative upper bound when $0 < \ell < 1$}. %In this regime note that the previous bound \eqref{20201221_14:08} with $(n, \Delta) = (1, \ell/2)$ still applies, yielding $\frac{4}{3}(\ell +1) + \frac{\ell^3}{12}  - \frac{\ell}{3}$. 
When $\ell$ is small, it is slightly better if we consider just one triangle. Let $\widehat{g_1} = (1+c) \, \widehat{K_{\Delta}}$ (the triangle of height $1+c$ and base $2 \Delta$). For $\frac{\ell}{2} < \Delta \leq 1$  and $c \geq \frac{\ell}{2 - \ell}$ such that 
\begin{equation}\label{20201221_14:59}
\frac{c}{1+c} = \frac{\ell/2}{\Delta}\,,
\end{equation}
this triangle contains a segment of length $\ell$ at height $1$. In other words, under \eqref{20201221_14:59}, we have the validity of \eqref{20201216_11:47} for $\xi_1 = \ell/2$. In this case, we have
\begin{equation}\label{20201221_16:24}
\rho(g_1) = (1+c) \left(1 + \frac{\Delta^2}{3}\right) = (1+c) \left(1 + \frac{\ell^2(1+c)^2}{12c^2}\right)\,,
\end{equation}
and we may minimize it over $c$. From calculus, we see that this amounts to solving a cubic polynomial, 
$$\left(1 + \frac{12}{\ell^2}\right) c^3 - 3c -2 =0.$$
This can be computed explicitly and yields a solution of the form
$$c = 6^{-1/3}\, \ell^{2/3} + o(\ell^{2/3}) \ \ \ ({\rm as} \ \ell \to 0).$$
For simplicity, we take 
$$c =\max\left\{ 6^{-1/3}\, \ell^{2/3}\,,\, \frac{\ell}{2 - \ell}\right\}.$$ 
Plugging this choice of $c$ in  \eqref{20201221_16:24} leads to the remaining upper bound stated in \eqref{20210108_17:51}.

\medskip

\noindent {\it Step 3. Lower bound}. The quantity appearing in \eqref{20201216_12:14} for $K_{\Delta}$ is
\begin{equation}\label{20201223_11:24}
2K_{\Delta}(0) - \rho(K_{\Delta}) = 2\Delta - 1 - \frac{\Delta^2}{3}.
\end{equation}
Hence, it is only profitable to include a triangle $\widehat{K_{\Delta}}$ in our configuration if the quantity in \eqref{20201223_11:24} is non-negative, that is, if $\Delta \geq 3 - \sqrt{6} = 0.5505\ldots$. If $0 < \ell < 2$ we just choose $\widehat{g_1} = \widehat{K_{\ell/2}}$ and $\xi_1 = \ell/2$ in \eqref{20201216_12:13}, provided that $\ell/2 \geq 3 - \sqrt{6}$, otherwise we go with the trivial lower bound $0$.

\smallskip

If $\ell \geq 2$, the idea here is to consider $n$ big triangles and (possibly) one small triangle at the end to adjust for the fractional part of $\ell$. We let $n = \lfloor \ell \rfloor -1$ and $\Delta = (1 + \{\ell\} )/2$. Observe then that $n + 2\Delta = \ell$. In the setup of extremal problem (EP2), we consider a configuration with $N = n$ or $n+1$ functions given by $\widehat{g_1} = \widehat{g_2} = \ldots = \widehat{g_{n}} =  \widehat{K_{1}}$ and $\widehat{g_{n+1}} = \Delta \widehat{K_{\Delta}}$, with $\xi_j = j$ for $j = 1,2,\ldots, n$ and $\xi_{n+1} = n + \Delta$, where the last pair $(\widehat{g_{n+1}}, \xi_{n+1})$ is only included if $\Delta \geq 3 - \sqrt{6}$. Observe that \eqref{20201216_12:13} is verified, and this configuration yields our desired lower bound
\begin{equation*}
\sum_{j = 1}^N \big(2g_j(0) - \rho(g_j)\big) = \frac{2}{3}n + \Delta\left(2\Delta - 1 - \frac{\Delta^2}{3}\right)_{\!+}.
\end{equation*}
\end{proof}

Observe that, when $\ell$ is large, the effective upper bound in Theorem \ref{Prop_20201220_16:37} with the multiplying factor $4/3$ is very close to the conceptual threshold \eqref{20210111_11:56} for the extremal problem (EP1) restricted to $\mathcal{A}_0$, and almost yields what we claim in Corollary \ref{Thm1_20201215}, but not quite there yet. We return to this point in \S \ref{Sub_Pf_Thm1}. As for the lower bound in Theorem \ref{Prop_20201220_16:37}, when $\ell$ is large, the multiplying factor $2/3$ is very close to the threshold \eqref{20210108_17:12} for the extremal problem (EP2) restricted to $\mathcal{A}_0$.

\subsection{Dirichlet kernels} \label{SS_Dir} We now discuss the reach of the extremal problem (EP3) in the setup of Corollary \ref{Cor_20210108}. The case when the lower endpoint $b$ is equal to $1$ is precisely the situation that is most useful when bounding the integral $J(\beta,T)$ in the next section. 

\subsubsection{Minima of Dirichlet kernels}\label{Subsec_20201221_09:26}For $n \in \Z_{\ge 0}$ we consider the {\it Dirichlet kernel} $D_n$ given by 
\begin{equation} \label{9_14_8:00pm}
	D_n(x) = \sum_{k=-n}^n e^{ikx} = 1+ 2 \sum_{k=1}^n \cos(kx) = \frac{\sin\big((n+1/2)x\big)}{\sin(x/2)}.
	\end{equation}
Let us define the minimum
	\begin{equation}\label{20200915_13:33}
	\mathfrak{m}(n) := \min_{\theta \in [0,2\pi]} \frac{\sin\big((2n+1)\theta\big)}{\sin \theta} = \min_{x\in\mathbb{R}} D_n(x)\,,
	\end{equation}
and the universal constant
\begin{equation}\label{20201221_09:30}
c_0 := \min_{x\in\mathbb{R}} \frac{\sin x}{x} = -0.21723\ldots.
\end{equation}
In Appendix A, we briefly verify the bounds
\begin{equation}\label{20210110_11:41}	
	2c_0  - \frac{(2 \pi -1)}{n}  \ \leq \ \frac{\mathfrak{m}(n)}{n} \ \leq \ 2c_0 + \frac{5.4935}{n}
\end{equation}
for $n \geq 1$, which in particular implies that 
\begin{equation*}
	\displaystyle\lim_{n\to\infty}\dfrac{\mathfrak{m}(n)}{n}=2c_0.
\end{equation*}
Hence, the moral is that Dirichlet kernels cannot be too negative when compared to their maximal value (attained at the origin). One of the main insights here is how to properly take advantage of that information in our context.

\subsubsection{A max-min optimization} We establish the following effective upper and lower bounds for the integral of $F(\alpha)$ in the interval $[1,\beta]$. Our lower bound is stated in terms of the minima $\mathfrak{m}(n)$ and, although our main focus is the behavior for large $\beta$, we try also to be careful for small values of $\beta$. In the argument below, we choose the degree of the Dirichlet kernel in order to optimize the effect that the minimum $\mathfrak{m}(n)$ is not too negative.

\begin{theorem}[Symmetric bounds] \label{17_25}
	Assume RH and let $\beta >1$. Let $c_0$ be given by \eqref{20201221_09:30} and $\mathcal{C}^+_{\blacktriangle}$ given by \eqref{20210108_17:51}. Then, as $T\to\infty$, we have
\begin{equation}\label{20210110_11:53}
	\mathcal{C}^-(1, \beta)  + o(1) \le \int_1^{\beta} F(\alpha,T) \, \d\alpha \le \mathcal{C}^+(1, \beta) +o(1),
\end{equation}	
where
\begin{equation}\label{20201116_16:33}
\mathcal{C}^+(1, \beta) = \frac{\mathcal{C}^+_{\blacktriangle}(2\beta)}{2}   - 1
\end{equation}
and
\begin{align}\label{20201116_16:56}
\mathcal{C}^-(1, \beta) = \max_{n \in \N} G_n(\beta) \ \geq \ \left(1 + \frac{c_0}{3}\right) (\lfloor\beta\rfloor-1) -\frac{(\pi+1)}{3}.
\end{align}
Here the functions $\{G_n\}_{n \in \N}$ are given by
\begin{align}\label{20201120_16:51}
\ \ G_1(\beta) = \min\left\{  \left(\beta + \tfrac{2}{3 \beta} -2\right)_+\, , \, \tfrac{1}{3}\right\};
\end{align}
\begin{align}\label{20201120_16:40}
G_n(\beta) = \left(n-\frac{1}{2}\right)\min\left\{1, \frac{\beta}{n}\right\} +  \mathfrak{m}(n-1) \left(\frac{\min\big\{1, \frac{\beta}{n}\big\}^2}{6} - \frac{\min\big\{1, \frac{\beta}{n}\big\}}{2}+\frac{1}{2}\right) -1 \ \ \ (n \geq 2).
\end{align}
\end{theorem}
Before moving to the proof of this result, let us make a few comments. Observe that when $2\beta$ is integer, the constant in \eqref{20201116_16:33} is reduced to 
\begin{equation*} 
\mathcal{C}^+(1, \beta)  = \frac{4}{3}(\beta -1) +\frac{7}{8}.\end{equation*}
We have also already observed that the function $\beta \mapsto \mathcal{C}^+(1, \beta)$ is continuous and non-decreasing. Note that \eqref{20210110_11:41} guarantees that the maximum in \eqref{20201116_16:56} is attained for some $n \leq 13 \,\beta$ (from that value on we actually have $G_n(\beta) \leq 0$). In particular, the function $\beta \mapsto \mathcal{C}^-(1, \beta)$ is, locally, a maximum of a finite number of continuous functions, hence it is also continuous. It is also clear that $\beta \mapsto \mathcal{C}^-(1, \beta)$ is non-decreasing. The particular choice $n = \lfloor \beta \rfloor \geq 3$ in \eqref{20201116_16:56} (which is generally near-optimal) gives us the effective lower bound
\begin{equation*}
\mathcal{C}^-(1, \beta) \geq \lfloor \beta \rfloor -\frac{3}{2} + \frac{\mathfrak{m}(\lfloor\beta\rfloor-1)}{6} \geq \left(1 + \frac{c_0}{3}\right) (\lfloor\beta\rfloor-1) -\frac{(\pi+1)}{3},
\end{equation*}
stated in \eqref{20201116_16:56}. Note the use of \eqref{20210110_11:41} in the last inequality above. Observe that, for large $\beta$, the multiplying factor 
$$1 + \frac{c_0}{3} = 0.92758\ldots$$ 
on the right-hand side of \eqref{20201116_16:56} is only slightly short of the conjectured value of $1$ in (II), and is one of the highlights of this theorem. The first few values of $\mathfrak{m}(n)$ are 
$$\mathfrak{m}(0)=1\ ;\ \mathfrak{m}(1)=-1\ ;\ \mathfrak{m}(2)= - \frac{5}{4} \ ;\  \mathfrak{m}(3)= -\frac{14\sqrt 7 + 7}{27} = -1.63113\ldots \ ; \ \mathfrak{m}(4)=-2.03911\ldots$$ 
%$$\mathfrak{m}(5)=-2.45753\ldots\ ;\ \mathfrak{m}(6)= -2.88117\ldots\ ;\  \mathfrak{m}(7)= -3.30782 \ldots \ ; \ \mathfrak{m}(8)=-3.73637\ldots.$$
Our lower bound $\mathcal{C}^-(1, \beta)$ starts to be non-trivial at $\beta_1 = 1.57735\ldots$ and from that value up to $\beta_2 = 1.77243\ldots$ the maximum in \eqref{20201116_16:56} is attained when $n=1$. From $\beta_2$ up to $\beta_3 = 3.02404\ldots$ the maximum is attained when $n=3$. From $\beta_3$ up to $\beta_4 = 4.04983\ldots$ the maximum is attained when $n=4$ and so on. In particular, we have
$$\mathcal{C}^-(1,2) = \frac{79}{216} = 0.36574\ldots \ ; \ \mathcal{C}^-(1,3) = \frac{31}{24} = 1.29166\ldots \ ; \ \mathcal{C}^-(1,4) = 2.22814\ldots$$
See Figure \ref{figure_C_beta} for the graphs of $\beta \mapsto \mathcal{C}^+(1, \beta)$ and $\beta \mapsto \mathcal{C}^-(1, \beta)$ for small values of $\beta$. 

\begin{figure} 
\includegraphics[scale=0.6]{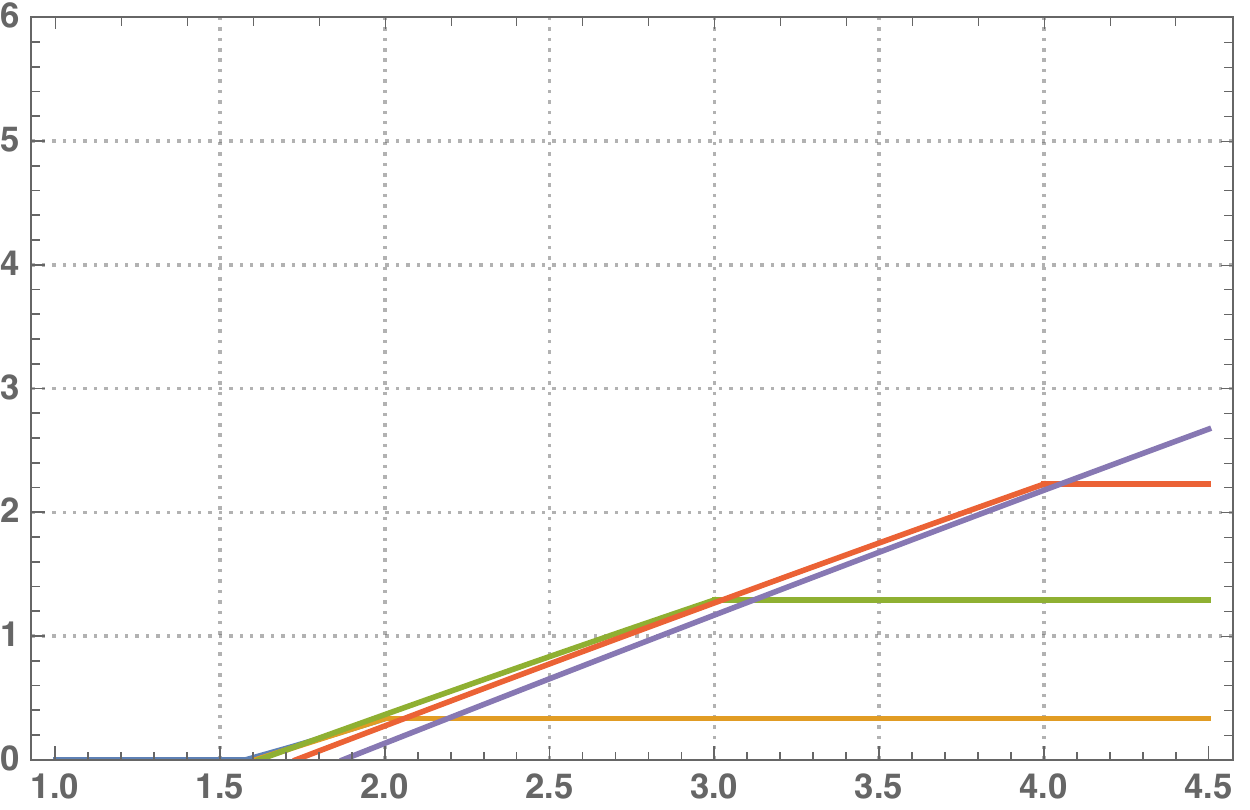}  \qquad 
\includegraphics[scale=0.6]{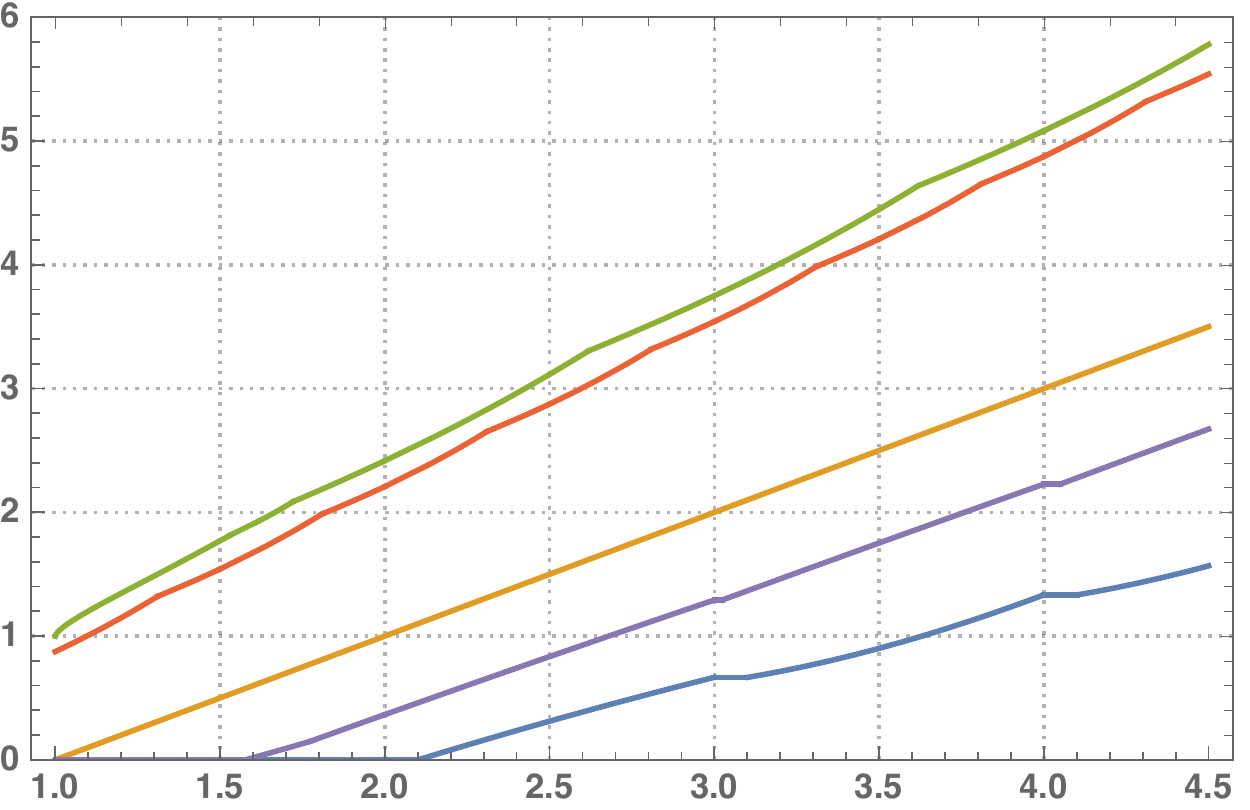}
\caption{On the left, the competition between the lower bounds $G_n$ for $n=1,2,3,4,5$. On the right, the plots of $\beta \mapsto \mathcal{C}^+(1, \beta)$ (in red), $\beta \mapsto \mathcal{C}^-(1, \beta)$ (in purple) and the conjectured asymptotic $\beta \mapsto \beta -1$ (in orange), for small values of $\beta$. In this symmetric setup, these always do better than the triangle bounds $\beta \mapsto \mathcal{C}^+_{\blacktriangle}(\beta -1)$ (in green) and $\beta \mapsto \mathcal{C}^-_{\blacktriangle}(\beta -1)$ (in blue) coming from Theorem \ref{Prop_20201220_16:37}.}
\label{figure_C_beta}
\end{figure}

\smallskip

We remark that, in this symmetric setup, the bounds coming from Theorem \ref{17_25} are better than the triangle bounds coming from Theorem \ref{Prop_20201220_16:37}, that is, for all $\beta >1$, one has
\begin{equation}\label{20210115_10:59}
\mathcal{C}^+(1, \beta) = \frac{\mathcal{C}^+_{\blacktriangle}(2\beta)}{2}   - 1 \leq \mathcal{C}^+_{\blacktriangle}(\beta-1)
\end{equation}
and
\begin{equation}\label{20210115_10:50}
\mathcal{C}^-(1, \beta) \geq \mathcal{C}^-_{\blacktriangle}(\beta-1).
\end{equation}
Inequality \eqref{20210115_10:59} is a routine explicit computation. Inequality \eqref{20210115_10:50} follows from \eqref{20201116_16:56} for large $\beta$ (say, for $\beta \geq 12$) and for small $\beta$ we verify it numerically. Figure \ref{figure_C_beta} also illustrates this dominance.

\smallskip

\noindent {\sc Remark}: In the small range $\frac{3}{2} < \beta < \frac{11}{7}$, we note that Radziwi\l\l \ \cite{R} obtains, with different methods, the lower bound
\[
\int_1^{\beta} \widetilde{F}(\alpha,T) \, \mathrm{d}\alpha \ge \beta - \frac{3}{2} + o(1),
\]
as $T \to \infty$, for the integral of the variant $\widetilde{F}(\alpha,T)$ defined in \eqref{20210128_08:40}.

\begin{proof}[Proof of Theorem \ref{17_25}] The upper bound in \eqref{20210110_11:53} plainly follows from Corollary \ref{Cor_20210108} and \eqref{20210110_11:54}.

\smallskip

For the lower bound, first let $n \in \N$, $n \geq 2$ and let $\Delta = \min\{1, \beta/n\}$. In the setup of extremal problem (EP3), we consider a configuration with $N = 2n-1$ functions given by $\widehat{g_1} = \widehat{g_2} = \ldots = \widehat{g}_{2n-1} =  \widehat{K_{\Delta}}$; $\eta_j = (n - j)\Delta$ for $j = 1,2, \ldots, 2n-1$; and $\frak{r}_1 = \frak{r}_2 = \ldots = \frak{r}_{2n-1} = \frak{r}$ given by
\begin{align}\label{20210112_23:07}
\frak{r} = \inf_{\substack{x \in \R \\ K_{\Delta}(x) \neq 0}} \frac{K_{\Delta}(x)\  {\rm Re} \left( \sum_{j=1}^{2n-1} e^{2\pi i (n-j)\Delta x}\right)}{(2n-1) K_{\Delta}(x)} = \min_{x \in \R} \frac{  D_{n-1}(2\pi \Delta x)}{(2n-1)} = \frac{\frak{m}(n-1)}{2n-1}.
\end{align}
Definition \eqref{20210112_23:07} assures the validity of \eqref{20210108_10:11}. Observe also that \eqref{20210107_21:13} is verified (with $b = -\beta$), that is 
\begin{align}\label{20210111_09:25}
\sum_{j=1}^{2n-1}  \widehat{K_{\Delta}}(\alpha - (n-j)\Delta) = \sum_{k=-(n-1)}^{n-1}  \widehat{K_{\Delta}}(\alpha+k\Delta) \leq  \chi_{[-\beta,\beta]}(\alpha)
\end{align}
for all $\alpha \in \R$. Therefore, recalling \eqref{20210111_09:07}, the outcome appearing in \eqref{20210111_09:03} for this particular configuration is
\begin{align*}
\sum_{j=1}^{2n-1} \big(g_j(0)  + \frak{r} \big(\rho(g_j) - g_j(0)\big)\big) = (2n -1)\Delta  + \frak{m}(n-1)\left(1 + \frac{\Delta^2}{3} - \Delta\right) \leq \mathcal{W}_*^{-}(-\beta, \beta).
\end{align*}
Dividing by $2$ and subtracting $1$ we have
\begin{align}\label{20210111_09:19}
G_n(\beta):=\left( n - \frac12\right) \Delta + \frak{m}(n-1)\left(\frac12 + \frac{\Delta^2}{6} - \frac{\Delta}{2}\right) - 1 \leq \frac{\mathcal{W}_*^{-}(-\beta, \beta)}{2} - 1\,,
\end{align}
and the lower bounds with each of these functions $G_n(\beta)$, for $n \geq 2$, follow from Corollary \ref{Cor_20210108}. We can now optimize the choice of the parameter $n$ here. Note that $n=1$ would have given a negative value for the term on the left-hand side of \eqref{20210111_09:19}, and that is the reason we are not considering it for the moment. Therefore, we define the function $G_1$ differently.

\smallskip

When $1< \beta \leq 2$ and we consider $n =2$ in the configuration above, note that $\Delta = \beta /2$ and we may replace \eqref{20210111_09:25} by the slightly stronger inequality
$$\left(\sum_{k=-1}^{1}  \widehat{K_{\Delta}}(\alpha+k\Delta)\right) +  \frac{\big(|\alpha| - \Delta\big)}{\Delta}\, \chi_{\{\Delta \leq |\alpha|\leq 1\}}(\alpha) \leq \chi_{[-\beta,\beta]}(\alpha).$$
Following the same computation as in \eqref{20210111_09:32} for the integral of $F(\alpha)$ from $-\beta$ to $\beta$, using \eqref{F formula} and the fact that $\frak{r} = \mathfrak{m}(1)/3=-1/3$, we would then obtain
$$\int_1^{\beta} F(\alpha) \, \mathrm{d} \alpha \geq 2\Delta + \frac{1}{3\Delta} -2 + o(1) = \beta + \frac{2}{3\beta} -2 + o(1).$$
This is the function we called $G_1(\beta)$ in \eqref{20201120_16:51} (technically speaking, its non-negative part). This concludes the proof of Theorem \ref{17_25}.
\end{proof}

For completeness, we record here the most refined explicit versions of upper and lower bounds for the integral of $F(\alpha)$ over a generic interval, by combining Theorems \ref{Prop_20201220_16:37} and \ref{17_25}.

\begin{corollary}\label{Cor_20210111_13:32}
Assume RH, let $\beta > b >1$ and set $\ell= \beta - b$. Then, as $T\to\infty$, we have
\begin{equation*}
	\mathcal{C}^-(b,\beta)  + o(1) \le \int_b^{\beta} F(\alpha,T) \, \d\alpha \le \mathcal{C}^+(b,\beta) +o(1),
\end{equation*}	
where
\begin{equation}\label{20201119_10:35}
\mathcal{C}^-(b,\beta)  = \max\left\{\mathcal{C}^-_{\blacktriangle}(\ell)\, , \,   \mathcal{C}^-(1,\beta) - \mathcal{C}^+(1, b)  \right\}
\end{equation}
and
\begin{equation}\label{20201119_10:36}
\mathcal{C}^+(b,\beta)  = \min\left\{\mathcal{C}^+_{\blacktriangle}(\ell)\, , \, \mathcal{C}^+(1,\beta) - \mathcal{C}^-(1, b)\right\}.
\end{equation}
\end{corollary}
\begin{proof}
The triangle bounds come directly from Theorem \ref{Prop_20201220_16:37}, while the identity
\begin{align}\label{20210113_09:52}
\int_b^{\beta} F(\alpha) \, \mathrm{d} \alpha =  \int_1^{\beta} F(\alpha) \, \mathrm{d} \alpha - \int_1^{b} F(\alpha) \, \mathrm{d} \alpha
\end{align}
allows us to use the symmetric bounds from Theorem \ref{17_25}.
\end{proof}
Note that the bounds $\mathcal{C}^{\pm}(b,\beta)$ in \eqref{20201119_10:35} and \eqref{20201119_10:36} are continuous functions of two variables. For a fixed $b > 1$, the lower bound in \eqref{20201119_10:35} is going to be $\mathcal{C}^-(1,\beta) - O(1)$ for large $\beta$. As observed in \eqref{20201116_16:56}, this comes with a multiplying factor of
$$1 + \frac{c_0}{3} = 0.92758\ldots$$
which is almost what we claim in Corollary \ref{Thm1_20201215} but, technically speaking, not quite there yet. We return to this point in the next subsection.

\subsection{Proof of Theorem \ref{Thm1_20201215*} and Corollary  \ref{Thm1_20201215}}  \label{Sub_Pf_Thm1}

\subsubsection{Upper bound} \label{PfThm1_UB} A natural idea to deal with the asymptotic upper bound is to morally consider, in the formulation of (EP1), copies a single function $g$.  Let $\mathcal{A}_1 \subset \mathcal{A}$ be the subclass of bandlimited functions in $\mathcal{A}$, i.e. the functions $g \in \mathcal{A}$ such that $\widehat{g}$ has compact support. Note that $\mathcal{A}_0 \subset \mathcal{A}_1 \subset \mathcal{A}$. For each $g \in \mathcal{A}_1$, we define its periodization on the Fourier side
\begin{equation*}
P_{\widehat{g}}(\alpha) := \sum_{n \in \Z}\widehat{g}(\alpha+n).
\end{equation*}
This is a continuous and 1-periodic function, and 
\begin{equation}\label{20210204_10:48}
\int_{0}^1 P_{\widehat{g}}(\alpha)\,\d\alpha = \int_{\R} \widehat{g}(\alpha)\,\d\alpha = g(0) \geq 0.
\end{equation}
For our next extremal problem, it is convenient to restrict matters to the subclass $\mathcal{A}_1$.

\subsubsection*{Extremal problem 4 {\rm (EP4)}} Find the infimum
\begin{equation}\label{20210112_11:26}
{\bf C}^+ := \inf_{0 \neq g \in \mathcal{A}_1} \ \frac{\rho(g)}{\displaystyle \min_{0 \leq \alpha \leq 1} \big| P_{\widehat{g}}(\alpha) \big|}.
\end{equation}

Let us see how this fits into our framework of problem (EP1).  Let $0 \neq g \in \mathcal{A}_1$, and assume that 
\begin{equation}\label{20210204_10:50}
\min_{0 \leq \alpha \leq 1} \big| P_{\widehat{g}}(\alpha) \big| \neq 0.
\end{equation}
Since $P_{\widehat{g}}$ is 1-periodic and continuous, from \eqref{20210204_10:48} and \eqref{20210204_10:50} we must have $P_{\widehat{g}}(\alpha) >0$ for all $0 \leq \alpha \leq 1$. By multiplying $g$ by an appropriate constant (note that the ratio in \eqref{20210112_11:26} is invariant under such operation), we may hence assume that 
\begin{equation}\label{20210112_11:29}
\min_{0 \leq \alpha \leq 1} \big| P_{\widehat{g}}(\alpha) \big|  = \min_{0 \leq \alpha \leq 1}  P_{\widehat{g}}(\alpha) =1.
\end{equation}
Assume that ${\rm supp}(\widehat{g}) \subset [-M, M]$, where $M \in \N$. Given $\ell >0$ large, in the setup of (EP1) let $N = \lceil \ell \rceil + 2M - 1$ and consider the configuration given by $\widehat{g_1} = \widehat{g_2} = \ldots = \widehat{g_N} = \widehat{g}$ and $\xi_j = (j- M)$ for $j = 1,2,\ldots, N$. From the fact that $\widehat{g}$ is continuous and ${\rm supp}(\widehat{g}) \subset [-M,M]$, together with \eqref{20210112_11:29}, we have 
\begin{equation}\label{20210112_11:51}
\sum_{j =1}^N \widehat{g_j} (\alpha - \xi_j) = P_{\widehat{g}}(\alpha)  \geq 1
\end{equation}
for $0 \leq \alpha \leq \lceil \ell \rceil$ (and in particular for $0 \leq \alpha \leq \ell$). Every term in the sum on the left-hand side of \eqref{20210112_11:51} is zero if $\alpha \leq -2M +1 $ or $\alpha \geq \lceil \ell \rceil + 2M -1$, hence the sum itself is zero in this range. If the sum is non-negative in the remaining set $[-2M+1, 0] \cup \big[ \lceil \ell \rceil,  \lceil \ell \rceil +2M-1\big]$ we will have achieved \eqref{20201216_11:47}. There is, however, the possibility that the sum on the left-hand side of \eqref{20210112_11:51} is negative in some parts of the set $[-2M+1, 0] \cup \big[ \lceil \ell \rceil,  \lceil \ell \rceil +2M-1\big]$, but this is not going to be a big issue here, for in this case we can fix the situation in order to achieve \eqref{20201216_11:47} by further including in our configuration a finite number of triangles of the form $c\widehat{K_{1}}$, where the number of triangles and their height $c$ may depend on $g$, but not on $\ell$. We have then showed that 
$$ \mathcal{W}^+(\ell) \leq \ell \, \rho(g) + O(1),$$
where the constant in $O(1)$ may depend on $g$, but not on $\ell$. This implies that, for any fixed $\varepsilon >0$, we have
\begin{equation*}
\mathcal{W}^+(\ell) \leq \ell \, ({\bf C}^+ + \varepsilon)
\end{equation*}
for large $\ell$. Hence, for any fixed $\varepsilon >0$, from Theorem \ref{Prop_20201217_09:51} we have
\begin{equation*}%\label{20210113_09:36}
\int_b^{b + \ell} F(\alpha,T) \, \d\alpha \le \ell \, ({\bf C}^+ + \varepsilon) +o(1)
\end{equation*}
for large $\ell$, as $T \to \infty$. This establishes the upper bound proposed in Theorem \ref{Thm1_20201215*}.

\smallskip

Finding the exact value of the constant ${\bf C}^+$ seems to be a hard problem. At the moment we can provide a reasonable approximation by working within the subclass $\mathcal{A}_0 \subset \mathcal{A}_1$. If $g \in \mathcal{A}_0$, a classical result of Krein \cite[p.~154]{A} guarantees that $g(x) = |h(x)|^2$, where $h \in L^2(\R)$ and ${\rm supp}(\widehat{h}) \subset [-\tfrac12,\tfrac12]$. As we have seen in Theorem \ref{Prop_20201220_16:37}, a natural choice is $g(x) = K_1(x) = \big(\frac{\sin \pi x}{ \pi x}\big)^2$, for which $\widehat{g}(\alpha) = (1 - |\alpha|)_+$ has the triangular graph. This corresponds to the 2choice $\widehat{h}(\alpha) = \chi_{[-\frac12, \frac12]}(\alpha)$ in Krein's decomposition, and yields the outcome of $4/3$. We experimented with polynomial perturbations of low degree (up to $8$) of this function and the search routine provided some better options, for instance
$$\widehat{h}(\alpha) = \big( 10 + 2\alpha^2 - 35\alpha^4\big)\, \chi_{[-\frac12, \frac12]}(\alpha)\,,$$
which yields the outcome
\begin{equation*}
\frac{\rho(g)}{\displaystyle \min_{0 \leq \alpha \leq 1} \big| P_{\widehat{g}}(\alpha) \big|} =1.33017\ldots.
\end{equation*}
This establishes the rightmost inequality in \eqref{20210205_11:56} and hence the upper bound proposed in Corollary \ref{Thm1_20201215}.

\subsubsection{Lower bound} \label{PfThm1_LB} The idea here is similar, now considering copies of a suitable function $g \in \mathcal{A}$ in the centered formulation of (EP3). Let $g \in \mathcal{A}$ and assume that $g(0) >0$ (this assumption is harmless here since $g(0) =0$ would yield an undesirable negative numerator in the formulation \eqref{20210204_11:03} below). For $m \in \N$ we define
\begin{equation}\label{20210204_14:16}
K_m(g) := \max_{\alpha \in \R} \sum_{n=0}^m \widehat{g}(\alpha +n).
\end{equation}
Note that 
\begin{equation*}
\int_{-1}^{0}  \left(\sum_{n=0}^m \widehat{g}(\alpha +n)\right)\,\d\alpha = \int_{-1}^{m} \widehat{g}(\alpha) \,\d\alpha \geq \int_{-\infty}^{\infty} \widehat{g}(\alpha) \,\d\alpha = g(0).
\end{equation*}
Hence $K_m(g) \geq g(0) >0$. The fact that the maximum is indeed attained in \eqref{20210204_14:16} follows from the fact that the sum is continuous and goes to zero as $|\alpha| \to \infty$ (Riemann-Lebesgue lemma). We observe that $\{K_m(g)\}_{m \in \N}$ is a non-increasing sequence and set
\begin{equation*}
K(g) := \lim_{m\to \infty} K_m(g) \geq g(0) >0.
\end{equation*}
Let $c_0$ be the constant given by \eqref{20201221_09:30}. We consider the following extremal problem.

\subsubsection*{Extremal problem 5 {\rm (EP5)}} Find the supremum
\begin{equation}\label{20210204_11:03}
{\bf C}^- := \sup_{\substack{0 \neq g \in \mathcal{A} \\ g(0) > 0}} \  \frac{g(0) + c_0 \big( \rho(g) - g(0)\big)}{K(g)}.
\end{equation}

\smallskip

Let us see how this fits into the framework of (EP3). Let $0 \neq g \in \mathcal{A}$ with $g(0) >0$ and assume without loss of generality that $K(g) = 1$. Given $\delta >0$ small, let $m_0 = m_0(\delta)$ be such that
\begin{equation}\label{20210204_14:58}
1 \leq K_m(g)\leq 1+ \delta
\end{equation} 
for $m \geq m_0$. Let $\beta$ be large, in particular with $2 \lfloor \beta \rfloor \geq m_0 +2$, and set $n= \lfloor \beta \rfloor$. In the framework of (EP3) we let $N = 2n-1$ and consider the configuration given by $\widehat{g_1} = \widehat{g_2} = \ldots = \widehat{g}_{2n-1} = \widehat{g}/(1 + \delta)$\,;\, $\eta_j = (n - j)$ for $j = 1,2, \ldots, 2n-1$; and $\frak{r}_1 = \frak{r}_2 = \ldots = \frak{r}_{2n-1} = \frak{r}$ given by
\begin{align*}
\frak{r} = \inf_{\substack{x \in \R \\ g(x) \neq 0}} \frac{ \frac{g(x)}{1 + \delta}\  {\rm Re} \left( \sum_{j=1}^{2n-1} e^{2\pi i (n-j) x}\right)}{(2n-1) \, \frac{g(x)}{1 + \delta}} = \min_{x \in \R} \frac{  D_{n-1}(2\pi x)}{(2n-1)} = \frac{\frak{m}(n-1)}{2n-1}.
\end{align*}
This assures the validity of \eqref{20210108_10:11}. From the fact that $g \in \mathcal{A}$ (in particular, the condition $\widehat{g}(\alpha) \leq 0$ for $|\alpha| \geq 1$), together with \eqref{20210204_14:16} and \eqref{20210204_14:58}, one can verify \eqref{20210107_21:13} (with $b = -\beta$). For this configuration, the outcome appearing in \eqref{20210111_09:03} yields
\begin{equation*}
 \frac{(2n-1)}{1 + \delta} \left( g(0) + \frac{\frak{m}(n-1)}{2n-1} \big( \rho(g) - g(0)\big)\right) \leq \mathcal{W}_*^{-}(-\beta, \beta).
\end{equation*}
By using \eqref{20210110_11:41}, we arrive at the inequality
\begin{equation*}
\frac{\beta}{1 + \delta}  \left(g(0) + c_0 \big( \rho(g) - g(0)\big) \right) - O(1) \leq \frac{\mathcal{W}_*^{-}(-\beta, \beta)}{2}\,,
 \end{equation*}
where the constant in $O(1)$ may depend on $g$, but not on $\beta$. Therefore, for any fixed $\varepsilon >0$, we have 
\begin{equation*}
\beta  \left({\bf C}^- \!- \varepsilon \right)  \leq \frac{\mathcal{W}_*^{-}(-\beta, \beta)}{2}
 \end{equation*}
 for large $\beta$. Hence,  for any fixed $\varepsilon >0$ and $b \geq 1$, from Corollary \ref{Cor_20210108} and a decomposition as in \eqref{20210113_09:52} we have 
\begin{equation*}%\label{20210113_10:31}
\ell \, ({\bf C}^- - \varepsilon) +o(1) \leq 
 \int_b^{b + \ell} F(\alpha,T) \, \d\alpha 
\end{equation*}
for large $\ell$, as $T \to \infty$. This establishes the lower bound proposed in Theorem \ref{Thm1_20201215*}.

\smallskip

As in the extremal problem (EP4), the precise value of the constant ${\bf C}^-$ is unknown to us but we can provide a reasonable approximation by working within the subclass $\mathcal{A}_0 \subset \mathcal{A}$. In this case, note that $K_m(g) = K_1(g)$ for all $m \in \N$. As argued before, if $g \in \mathcal{A}_0$, Krein's decomposition \cite[p.~154]{A} guarantees that $g(x) = |h(x)|^2$, where $h \in L^2(\R)$ and ${\rm supp}(\widehat{h}) \subset [-\tfrac12,\tfrac12]$. We have seen in Theorem \ref{17_25} that the choice $g(x) = \big(\frac{\sin \pi x}{ \pi x}\big)^2$, corresponding to $\widehat{h}(\alpha) = \chi_{[-\frac12, \frac12]}(\alpha)$, yields the outcome
$$1 + \frac{c_0}{3} = 0.92758\ldots.$$
Experimenting with polynomials perturbations of low degree (up to $8$) of this function, the search routine provided some slightly better options, for instance
$$\widehat{h}(\alpha) = \big( 5 - \alpha^2\big) \chi_{[-\frac12, \frac12]}(\alpha)\,,$$
which yields the outcome 
\begin{equation*}
\frac{g(0) + c_0 \big( \rho(g) - g(0)\big)}{K(g)} = 0.92781\ldots.
\end{equation*}
This establishes the leftmost inequality in \eqref{20210205_11:56} and hence the lower bound proposed in Corollary \ref{Thm1_20201215}.

\subsection{Limitations to mollifying $\zeta(s)$ on the critical line.}\label{Rad_sec} We now comment on an application of our explicit bounds for $F(\alpha)$. Following Radziwi\l\l \ \cite{R}, let 
\[
\mathcal{I}(M_\theta):= \frac{1}{T} \int_T^{2T} \big| 1 -\zeta(\tfrac{1}{2}+it) M_\theta(\tfrac{1}{2}+it) \big|^2 \, \mathrm{d}t, \quad \text{where } \ M_\theta(s) = \sum_{n\le T^\theta} \frac{a(n)}{n^s} 
\]
is a Dirichlet polynomial with $a(1)=1$ and $a(n) \ll_\varepsilon n^\varepsilon$ for all $\varepsilon>0$. For a fixed $\theta>0$, an important problem in the theory of the zeta function is to choose $M_\theta(s)$ so that $\mathcal{I}(M_\theta)$ is as small as possible, e.g.~ \cite{C,L}. In \cite[Theorem 1]{R}, it is shown that there is an absolute constant $c>0$ such that
\begin{equation} \label{Rad2}
\mathcal{I}(M_\theta) \ge \frac{c}{\theta}, 
\end{equation}
when $T$ is sufficiently large. When $\theta < \frac{1}{2}$, an unpublished argument of Soundararajan is presented which shows that $\mathcal{I}(M_\theta) \ge \frac{1}{\theta}+o(1)$, as $T\to\infty$. Assuming RH, Radziwi\l\l \ further connects the problem to the pair correlation of the zeros of $\zeta(s)$, by using a slight variant of our $F(\alpha)$ function, namely, 
\begin{align}\label{20210128_08:40}
\widetilde{F}(\alpha):=\widetilde{F}(\alpha,T) = \frac{2\pi}{T\log T} \sum_{T\leq \gamma,\gamma'\le 2T} T^{i \alpha (\gamma-\gamma')} w(\gamma-\gamma').
\end{align}
Under the additional assumption\footnote{In \cite{R}, the assumption is $a(p^k) \ll 1$ for primes $p$ and $k \in \mathbb N$, but it is sufficient to assume only the case $k = 1$ in Radziwi\l\l's proof.} that $a(p) \ll 1$ for primes $p$  %and $k\in\mathbb{N}$
, for fixed $\theta>0$ and sufficiently large $T$, \cite[Theorem 3]{R} gives
\begin{equation}\label{20210114_23:50}
\mathcal{I}(M_\theta) \ge \left( \frac{1}{2} + \int_1^{1+\theta+\varepsilon} \widetilde{F}(\alpha, T) \, \mathrm{d}\alpha \right)^{\!-1}
\end{equation}
assuming RH, where $\varepsilon>0$ is arbitrary. Note that when $\theta$ is large, under Montgomery's strong pair correlation conjecture, $c$ in \eqref{Rad2} can be taken to be $1^-$. Based upon these results, Radziwi\l\l \, suggests that the inequality \eqref{Rad2} holds with $c=1$ for all $\theta>0$.

\smallskip

With the alternative definition \eqref{20210128_08:40} we still have the validity of \eqref{inversion}, \eqref{F formula} and therefore \eqref{20201216_11:09}, and our framework yields the exact same bounds of \S \ref{Fourier_Opt}--\S \ref{Sub_Pf_Thm1} for the integral of $\widetilde{F}(\alpha)$ in bounded intervals. Relation \eqref{20210114_23:50} immediately leads us to the following corollary of Theorem \ref{17_25}.

\begin{corollary}\label{Rad} Assume RH. For fixed $\theta>0$ and $M_\theta(s)$ as above, assume also that $a(p) \ll 1$ for primes $p$. Then, as $T\to\infty$, we have
\[
\mathcal{I}(M_\theta) \ge \left( \frac{1}{2} \,+\, \mathcal{C}^+\big(1,1 + \theta\big) \right)^{\!-1} + o(1).
\]
\end{corollary}
When $\theta$ is large, from \eqref{20210114_23:50} and the discussion in \S \ref{PfThm1_UB} we see that, under RH and $a(p) \ll 1$, the value of $c$ in \eqref{Rad2} can be taken to be constant less than $1/{\bf C}^+$. We have seen that 
$$1/{\bf C}^+ > 1/(1.3302) > 0.7517,$$
which is close to the conjectured bound of 1.

\section{Primes in short intervals} \label{Sec_PSI_new}

\subsection{Sunrise approximations to the Fej\'{e}r kernel}\label{Sec_Sunrise_Approx}
In this subsection we develop some preliminaries for the upcoming discussion on the integral $J(\beta,T)$. The following extremal problem in analysis is going to be relevant for our purposes. \subsubsection*{Extremal problem 6 {\rm (EP6)}} Construct continuous functions $g^\pm: [0,\infty) \to \mathbb{R}$ verifying:
\begin{enumerate}
	\item[(i)] $g^\pm$ are non-increasing;
	\item[(ii)] $0 \leq g^-(x)\leq \bigg(\dfrac{\sin x}{x}\bigg)^2 \le g^+(x)$ for all $x\geq0$;
	\item[(iii)] $\int_{0}^{\infty}g^-(x)\,\dx$ is as large as possible and $\int_{0}^{\infty}g^+(x)\,\dx$ is as small as possible.
\end{enumerate}
This problem admits unique solutions with the functions $g^\pm$ constructed as follows. We have
\[
g^-(x) = \left\{ \begin{array}{cl}
\left(\dfrac{\sin x}{x}\right)^2, &\mbox{if $0 \leq x \le \pi$;} \\
0, &\mbox{if  $x \ge \pi $,}
\end{array} \right.
\]
with 
\begin{align}\label{20190715_05:44}
{ \bf L^-}:=\dfrac{\displaystyle\int_{0}^{\infty}g^-(x)\,\dx}{\displaystyle\int_{0}^{\infty}\bigg(\dfrac{\sin x}{x}\bigg)^2\dx} = \dfrac{2}{\pi}\displaystyle\int_{0}^{\pi} \bigg(\dfrac{\sin x}{x}\bigg)^2 \,\dx = 0.9028\ldots.
\end{align}
The construction of $g^+$ is as follows. Let $0=m_0 < m_1 < m_2 < m_3 <\ldots$ be the sequence of local maxima of $(\sin x/x)^2$ in $[0,\infty)$.  For each $k\ge 1$, let $a_k \in (m_{k-1},m_k)$ be such that $(\sin a_k/a_k)^2=(\sin m_k/m_k)^2$ (note that such $a_k$ indeed exists). Then $g^+$ is defined by
\[
g^+(x) = \left\{ \begin{array}{cl}
\left(\dfrac{\sin x}{x}\right)^2, &\mbox{if $x \in [m_{k-1},a_k)\, , \, k\geq 1$;} \\
\left(\dfrac{\sin m_k}{m_k}\right)^2, &\mbox{if  $x \in [a_k,m_k)\, , \, k\geq 1$,}
\end{array} \right.
\]
(see Figure \ref{figure3}) and a numerical verification yields
\begin{align}\label{20190705_11:49am}
{\bf L^+}:=\dfrac{\displaystyle\int_{0}^{\infty}g^+(x)\,\dx}{\displaystyle\int_{0}^{\infty}\bigg(\dfrac{\sin x}{x}\bigg)^2\dx} = \dfrac{2}{\pi}\displaystyle\int_{0}^{\infty}g^+(x)\,\dx = 1.0736\ldots.
\end{align}
The idea to consider this pair of functions is inspired in the classical sunrise lemma in harmonic analysis. When the sun rises over the graph of the Fej\'{e}r kernel from the right (resp. from the left) the visible portion is $g^+$ (resp. $g^-$). Throughout this section we reserve the notation $g^{\pm}$ for these sunrise approximations, and ${\bf L}^{\pm}$ for the constants in \eqref{20190715_05:44} and \eqref{20190705_11:49am}.
\begin{figure} 
	\includegraphics[scale=.3]{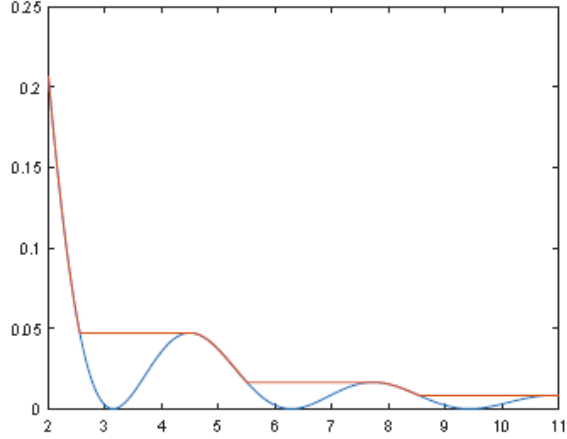}  
	\caption{Plots of $(\sin x /x)^2$ and $g^+(x)$ for $2 \le x\le 11$ }
	\label{figure3}
\end{figure}

%Note that $\frac{\mathrm{d}}{\mathrm{dx}} \big\{g^\pm(x) \big\} \le 0$ a.e. and that $g^+(x) \le \min\big\{1,\frac{1}{x^2}\big\}$.

\subsection{Asymptotic inequalities} The following lemma is a modification of Goldston \cite[Lemma 2]{G}, replacing the assumption of asymptotic relations in that paper by inequalities in the present setting. The sunrise approximations $g^\pm(x)$, from \S \ref{Sec_Sunrise_Approx}, play important roles in the proof below. 

\begin{lemma} \label{31_08_2:02am} Let $f:[0,\infty)\times [2,\infty)\to\R$ be a non-negative continuous function such that $f(t,\eta)\ll\log^2(t+2)$. Let $K(T,\eta):=\int_{0}^{T}f(t,\eta)\,\dt$ and $c \geq 0$.
	\begin{enumerate}
		\item[(i)] Suppose that $K(T,\eta) \leq \big(c+o(1)\big)\,T$, as $T\to\infty$, uniformly for $\eta\log^{-3}\eta\leq T \leq \eta\log^{3}\eta$. Then
		$$
		\int_{0}^{\infty}\bigg(\dfrac{\sin(\kappa t)}{t}\bigg)^2f(t,\eta)\,\dt\leq \big(c+o(1)\big)\,\dfrac{ \pi}{2}\,{\bf L}^+\,\kappa
		$$
		as $\kappa\to 0$, for $\eta\asymp 1/\kappa$.
		
		\smallskip
		
		\item[(ii)] Suppose that $\big(c+o(1)\big)\,T \le K(T,\eta) \ll T $, as $T\to\infty$, uniformly for $\eta\log^{-3}\eta\leq T \leq \eta\log^{3}\eta$. Then
		$$
		\int_{0}^{\infty}\bigg(\dfrac{\sin(\kappa t)}{t}\bigg)^2f(t,\eta)\, \dt\ge \big(c+o(1)\big)\,\dfrac{ \pi}{2}\,{\bf L}^-\,\kappa
		$$
		as $\kappa\to 0$, for $\eta\asymp 1/\kappa$.
	\end{enumerate}
\end{lemma}

\begin{proof} We only prove part (i), as the proof of part (ii) follows the same outline. We suppose that $\eta\asymp 1/\kappa$ and divide the integral to be bounded into four ranges:
	\[
	\begin{split}
	\int_{0}^{\infty}\bigg(\dfrac{\sin(\kappa t)}{t}\bigg)^2f(t,\eta)\,\dt &=\int_{0}^{\eta\log^{-3}\eta} + \int_{\eta\log^{-3}\eta}^{\eta\log\eta} + \int_{\eta\log\eta}^{\eta\log^{3}\eta} + \int_{\eta\log^{3}\eta}^{\infty}
	\\
	&:=A_1+A_2+A_3+A_4.
	\end{split}
	\]
	The main contribution will come from $A_2$, while the integrals $A_1, A_3,$ and $A_4$ will contribute an error term.  Using the fact that $f(t,\eta)\ll\log^2(t+2)$, we have
	\[
	A_1= \kappa^2\int_{0}^{\eta\log^{-3}\eta}\bigg(\dfrac{\sin(\kappa t)}{\kappa t}\bigg)^2f(t,\eta)\, \dt \ll \kappa^2\int_{0}^{\eta\log^{-3}\eta}\log^2(t+2)\,\dt \ll\kappa^2\dfrac{\eta}{\log \eta} \ll \dfrac{\kappa}{\log \eta}
	\]
	and
	\[
	A_4 = \int_{\eta\log^{3}\eta}^{\infty}\bigg(\dfrac{\sin(\kappa t)}{t}\bigg)^2f(t,\eta)\,\dt \ll \int_{\eta\log^{3}\eta}^{\infty}\dfrac{\log^2 t}{t^2}\, \dt \ll \dfrac{1}{\eta\log\eta}\ll \dfrac{\kappa}{\log\eta}.
	\]
	Since $f$ is non-negative, we use integration by parts to get
	\[
	\begin{split}
	A_3 & = \int_{\eta\log\eta}^{\eta\log^{3}\eta}\bigg(\dfrac{\sin(\kappa t)}{t}\bigg)^2f(t,\eta)\,\dt \leq \int_{\eta\log\eta}^{\eta\log^{3}\eta}\dfrac{1}{t^2}\,(K(t,\eta))'\, \dt  
	\\
	&  = \dfrac{K(\eta\log^{3}\eta,\eta)}{(\eta\log^3\eta)^2} -\dfrac{K(\eta\log\eta,\eta)}{(\eta\log\eta)^2} + 2\int_{\eta\log\eta}^{\eta\log^{3}\eta}\dfrac{1}{t^3}\,K(t,\eta)\,\dt  \ll\dfrac{1}{\eta\log\eta}\ll \dfrac{\kappa}{\log\eta}.
	\end{split}
	\]
	We now analyze the contribution from the integral $A_2$. Using integration by parts, we have
	\begin{align*} 	 
	A_2  &= \kappa^2\int_{\eta\log^{-3}\eta}^{\eta\log\eta}\bigg(\dfrac{\sin(\kappa t)}{\kappa t}\bigg)^2f(t,\eta)\,\dt \leq \kappa^2\int_{\eta\log^{-3}\eta}^{\eta\log\eta}g^+(\kappa t)\,f(t,\eta)\,\dt 
	\\
	& = \kappa^2\int_{\eta\log^{-3}\eta}^{\eta\log\eta}\big(-g^+(\kappa t)\big)'K(t,\eta)\,\dt +O\left(\frac{\kappa}{\log\eta}\right) ,
	\end{align*}
	where we have used the fact that $g^+(x) \le \min\big\{1,\frac{1}{x^2}\big\}$ to estimate the error term above.
	Since $g^+$ is non-increasing and absolutely continuous, we get
	\[
	\begin{split}
	\kappa^2\int_{\eta\log^{-3}\eta}^{\eta\log\eta}\big(-g^+(\kappa t)\big)'K(t,\eta)\,\dt & \leq   \kappa^2 \int_{\eta\log^{-3}\eta}^{\eta\log\eta}\big(-g^+(\kappa t)\big)'\,t\,(c+o(1))\,\dt. 
	\end{split}
	\]
	Again using that $g^+(x) \le \min\big\{1,\frac{1}{x^2}\big\}$, an integration by parts yields
	\begin{align*}
	\kappa^2\int_{\eta\log^{-3}\eta}^{\eta\log\eta}\big(-g^+(\kappa t)\big)'\,t\,\dt & = \kappa^2\int_{\eta\log^{-3}\eta}^{\eta\log\eta}g^+(\kappa t)\,\dt + O\bigg(\dfrac{\kappa}{\log\eta}\bigg) \\
	& = \kappa^2\int_{0}^{\infty}g^+(\kappa t)\,\dt - \kappa^2\int_{0}^{\eta\log^{-3}\eta}g^+(\kappa t)\,\dt -\kappa^2\int_{\eta\log\eta}^{\infty} g^+(\kappa t)\,\dt + O\bigg(\dfrac{\kappa}{\log\eta}\bigg)  \\
	& = \kappa\int_{0}^{\infty}g^+(t)\,\dt + O\bigg(\dfrac{\kappa}{\log\eta}\bigg). 
	\end{align*}
	Combining estimates, the lemma follows. 
\end{proof}

\subsection{Relating primes in short intervals to pair correlation}\label{PrimeSection} Our next theorem gives an explicit relationship between the integral $J(\beta,T)$ and the integral of $F(\alpha)$ in bounded intervals. % as follows.
\begin{theorem} \label{Primes}
Assume RH and let $\beta > b >0$. Let \,${\bf L}^{-}$ \!and \, ${\bf L}^{+}$ be the constants defined in \eqref{20190715_05:44} and \eqref{20190705_11:49am}. Then, as $T\to\infty$, we have
\begin{align}\label{20201211_12:35}
\begin{split}
{\bf L}^- \left( \lim_{\varepsilon \to 0^+}\, \liminf_{\tau \to \infty} \int_{b + \varepsilon}^{\beta - \varepsilon} F(\alpha, \tau) \,\mathrm{d} \alpha + o(1)\right)\, & \frac{\log^2T}{T}  \leq J(\beta,T)- J(b,T) \\
& \le {\bf L}^+ \left( \lim_{\varepsilon \to 0^+}\, \limsup_{\tau \to \infty} \int_{b - \varepsilon}^{\beta + \varepsilon} F(\alpha, \tau) \,\mathrm{d} \alpha + o(1)\right)\, \frac{\log^2T}{T}.
\end{split}
\end{align}
\end{theorem}

\noindent {\sc Remark:} From \eqref{asymp} and \eqref{F formula} it should be clear that, when $0 < b \leq 1$, the lower endpoints in the integrals appearing in \eqref{20201211_12:35} can be taken to be $b$ $($instead of $b+ \varepsilon$ and $b- \varepsilon$, respectively$)$. For the lower bound when $0 < b < 1$ and the upper bound when $0 < b \leq 1$ this follows directly by \eqref{F formula}. For the lower bound when $b =1$, we estimate instead $J(\beta,T)- J(1-\delta,T)$ and then send $\delta \to 0$ using \eqref{asymp}.

\smallskip

From \eqref{asymp}, Theorem \ref{17_25}, Corollary \ref{Cor_20210111_13:32} and Theorem \ref{Primes} (including the remark thereafter) we immediately get the following corollary.
\begin{corollary}\label{Cor4}
Assume RH and let $\beta > 1$. Then, as $T\to\infty$, we have
\begin{equation}\label{20201211_14:57}
\left( {\bf L}^- \, \mathcal{C}^-(1, \beta) + \frac{1}{2} + o(1)\right) \frac{\log^2T}{T}  \leq J(\beta,T) \leq \left( {\bf L}^+ \, \mathcal{C}^+(1, \beta) + \frac{1}{2} + o(1)\right) \frac{\log^2T}{T}.
\end{equation}
In general, if $\beta > b >1$, as $T\to\infty$, we have
\begin{equation}\label{20201211_14:44}
\left( {\bf L}^- \, \mathcal{C}^-(b, \beta)  + o(1)\right) \frac{\log^2T}{T}  \leq J(\beta,T) - J(b, T) \leq \left( {\bf L}^+ \,\mathcal{C}^+(b, \beta)  + o(1)\right) \frac{\log^2T}{T}.
\end{equation}
\end{corollary}

Previously, assuming RH, Goldston and Gonek in \cite{GG} had proved that for any $b >0$ one has 
\begin{align*}   
(0.307 + o(1))\,\dfrac{\log^2T}{T}\leq J(b+2,T)-J(b,T)\leq (21.647 + o(1)) \, \dfrac{\log^2T}{T}
\end{align*}
as $T \to \infty$. As we already observed in the introduction, from this estimate one can deduce that, for large $\beta$, 
\begin{align*} 
(0.153 \beta+ o(1))\,\dfrac{\log^2T}{T}\leq J(\beta,T)\leq (10.824 \beta + o(1)) \, \dfrac{\log^2T}{T}
\end{align*}
(the lower bound actually holds for all $\beta > 1$). In direct comparison, \eqref{20201119_10:35}, \eqref{20201119_10:36} and \eqref{20201211_14:44} imply that 
\begin{align}\label{20210111_14:13}
\left( \frac{2}{3}\,{\bf L}^-  + o(1)\right) \frac{\log^2T}{T}  \leq J(b+2,T) - J(b, T) \leq \left(  \frac{15}{4}\,{\bf L}^+  + o(1)\right) \frac{\log^2T}{T},
\end{align}
as $T \to \infty$. The constants in \eqref{20210111_14:13} are $ \tfrac{2}{3}\,{\bf L}^- = 0.6018\ldots$ and $\frac{15}{4}\,{\bf L}^+ = 4.026\ldots$. For large $\beta$, inequality \eqref{20201211_14:57} in Corollary \ref{Cor4} implies that, in \eqref{MontgomeryBound},  $D^-$ can be taken to be any constant less than ${\bf L}^-\left(1 + \frac{c_0}{3}\right) = 0.8374\ldots$ while $D^+$ can be taken to be any constant greater than $\frac{4}{3}{\bf L}^+ = 1.431\ldots$. These values are substantially closer to the conjectured value $1$. We now establish the further small improvement proposed in Theorem \ref{Thm_20201215_01:27*} and Corollary \ref{Thm_20201215_01:27}.

\begin{proof}[Proof of Theorem \ref{Thm_20201215_01:27*} and Corollary \ref{Thm_20201215_01:27}] From Theorem \ref{Primes} and Theorem \ref{Thm1_20201215*}, we see that, for large $\beta$, the value $D^+$ in \eqref{MontgomeryBound} can be taken to be any constant greater than ${\bf L}^+ {\bf C}^+$. We have shown that ${\bf L}^+ {\bf C}^+ < {\bf L}^+ (1.3302) < 1.4283$. Similarly, Theorem \ref{Primes} and and Theorem \ref{Thm1_20201215*} show that the value $D^-$ in \eqref{MontgomeryBound} can be taken to be any constant less than ${\bf L}^- {\bf C}^-$. We have showed that ${\bf L}^- {\bf C}^- >  {\bf L}^- (0.9278) > 0.8376$. This completes the proof.

\end{proof}

\begin{proof}[Proof of Theorem \ref{Primes}]
We partially follow the idea developed by Goldston and Gonek in \cite{GG}. Throughout the proof let 
\begin{equation*}
0 \leq a_1 < a_2 < a_3 < a_4
\end{equation*} 
be fixed real numbers (that will be conveniently specialized later). We let $g:=g_{a_1,a_2,a_3,a_4}: \R \to \C$ be a Schwartz function verifying 
\begin{equation*}
 |\widehat{g}| \le 1  \ {\rm on } \  \R \ \ ; \ \ \mathrm{supp}(\widehat{g}) \subset [a_1,a_4] \ \ ; \ \ \widehat{g}\equiv 1  \  {\rm on}  \  [a_2,a_3].
\end{equation*}
Then, from definition \eqref{20200918_12:11}, we plainly see that 
\begin{equation}\label{1eq}
J(a_3,T)- J(a_2,T) \le \int_1^\infty \left( \psi\left(x + \frac{x}{T}\right)-\psi(x) - \frac{x}{T} \right)^2 \left| \,\widehat{g}\left(\frac{\log x}{\log T} \right) \right|^2 \frac{\mathrm{d} x}{x^2} \le J(a_4,T)- J(a_1,T).
\end{equation}
From \cite[Eq. (8)]{GG}, with $e^{2 \kappa}=1+\frac{1}{T}$, we have
\begin{align}\label{2eq}
\begin{split} 
& \int_1^\infty \left( \psi\left(x + \frac{x}{T}\right)-\psi(x) - \frac{x}{T} \right)^2  \left|\, \widehat{g}\left(\frac{\log x}{\log T} \right) \right|^2  \, \frac{\mathrm{d} x}{x^2} \\
& = \frac{2}{\pi} \log^2 T \int_0^\infty \left( \frac{\sin(\kappa t )}{t} \right)^2 \left( \left| \sum_\gamma g\left((t-\gamma)\frac{\log T}{2\pi}\right) \right|^2 + \left|\sum_\gamma g\left((\gamma - t)\frac{\log T}{2\pi}\right) \right|^2 \right) \,\mathrm{d} t + O(1/T).
\end{split}
\end{align}
The implicit constant in the error term above may, in principle, depend on the function $g$. From now on let us write
\begin{equation*}
f(t,\eta) :=  \left| \sum_\gamma g\left((t-\gamma)\frac{\log \eta}{2\pi}\right) \right|^2 + \left|\sum_\gamma g\left((\gamma - t)\frac{\log \eta}{2\pi}\right) \right|^2.
\end{equation*}
Using \cite[Eqs. (5.1), (5.2) and (5.3)]{G}\footnote{See also \cite[Eq. (7)]{GG}, where there seems to be a typo and the lower endpoint of the integral should be zero.} we get
\[
\int_0^T f(t,\eta)  \,\mathrm{d} t = 2 T \int_0^\infty F(\alpha, T)\, |\widehat{g}(\alpha)|^2 \,\mathrm{d} \alpha+o(T),
\]
uniformly for $\eta\log^{-3}\eta\leq T \leq \eta\log^{3}\eta$. In this range of $T$ and $\eta$, using our assumptions on $\widehat g$ and the fact that $F\ge 0$, we arrive at 
\begin{align} \label{19_29} 
\begin{split}
& \left(2 \liminf_{\tau \to \infty} \int_{a_2}^{a_3} F(\alpha, \tau) \,\mathrm{d} \alpha + o(1)\right)  T   \leq \left(2 \int_{a_2}^{a_3} F(\alpha, T) \,\mathrm{d} \alpha + o(1)\right)T  \\
& \ \ \ \ \ \ \ \ \ \ \ \ \ \ \ \ \ \ \ \ \ \ \ \ \  \le \int_0^T f(t,\eta)  \,\mathrm{d} t \\
&\ \ \ \ \ \ \ \ \ \ \ \ \ \ \ \ \ \ \ \ \ \ \ \ \  \leq \left(2 \int_{a_1}^{a_4} F(\alpha, T) \,\mathrm{d} \alpha + o(1)\right) T \leq \left(2 \limsup_{\tau \to \infty} \int_{a_1}^{a_4} F(\alpha, \tau) \,\mathrm{d} \alpha + o(1)\right) T.
\end{split}
\end{align}

\noindent {\it Upper bound}.  From the fast decay of $g$ and the classical estimate for the number of zeros in an interval, one can show that $f(t,\eta)\ll\log^2(t+2)$ (see, for instance, \cite[p. 618]{GG}). Then, by \eqref{19_29} and Lemma \ref{31_08_2:02am} (i), we obtain 
\begin{align} \label{20_17}
\int_{0}^{\infty}\bigg(\dfrac{\sin(\kappa t)}{t}\bigg)^2f(t,\eta)\,\dt\leq \left(2 \limsup_{\tau \to \infty} \int_{a_1}^{a_4} F(\alpha, \tau) \,\mathrm{d} \alpha + o(1)\right)\dfrac{ \pi}{2}\,{\bf L}^+\,\kappa
\end{align} 
as $\kappa\to 0$, for $\eta\asymp 1/\kappa$. Choosing $\eta=T$ in \eqref{20_17}, and combining with \eqref{1eq} and \eqref{2eq} (recall that $\kappa = \frac{1}{2T}( 1 + o(1))$) we get
\begin{equation*}
J(a_3,T)- J(a_2,T) \le \left( \limsup_{\tau \to \infty} \int_{a_1}^{a_4} F(\alpha, \tau) \,\mathrm{d} \alpha + o(1)\right){\bf L}^+\, \frac{\log^2T}{T}
\end{equation*}
as $T\to \infty$. At this point we can take $a_3 = \beta$, $a_2 = b$, $a_1 \to a_2^-$ and $a_4 \to a_3^+$ to conclude.

\medskip

\noindent {\it Lower bound}.  By \eqref{19_29} and Lemma \ref{31_08_2:02am} (ii) we have
\begin{align}\label{20200918_12:55}
\int_{0}^{\infty}\bigg(\dfrac{\sin(\kappa t)}{t}\bigg)^2f(t,\eta)\,\dt\geq \left(2 \liminf_{\tau \to \infty} \int_{a_2}^{a_3} F(\alpha, \tau) \,\mathrm{d} \alpha + o(1)\right)\dfrac{ \pi}{2}\,{\bf L}^-\,\kappa
\end{align} 
as $\kappa\to 0$, for $\eta\asymp 1/\kappa$. As before, choosing $\eta=T$ in \eqref{20200918_12:55} and combining with \eqref{1eq} and \eqref{2eq}, we get
\begin{equation*}
J(a_4,T)- J(a_1,T) \ge \left(\liminf_{\tau \to \infty} \int_{a_2}^{a_3} F(\alpha, \tau) \,\mathrm{d} \alpha + o(1)\right){\bf L}^-\, \frac{\log^2T}{T}
\end{equation*}
as $T \to \infty$. We now take $a_4 = \beta$, $a_1 = b$, $a_2 \to a_1^+$ and $a_3 \to a_4^-$ to conclude.

\end{proof}

\section{The second moment of the logarithmic derivative of $\zeta(s)$}  

\subsection{Preliminaries} We start by presenting some auxiliary tools for the upcoming proof of Theorem \ref{log zeta}.

\subsubsection{Relating $I(a,T)$ to the Poisson kernel} Our starting point for the proof of Theorem \ref{log zeta} is a result of Goldston, Gonek, and Montgomery which, assuming RH, relates the integral $I(a,T)$ to the Poisson kernel
\begin{equation}\label{Def_Poisson_kernel_beta}
h_{b}(x) := \frac{b}{b^2 + x^2}.
\end{equation}
\begin{lemma} \label{GGM_Lemma}
Assume RH and let $0<a\le \sqrt{\log T}$. Then
\[
\begin{split}
I(a,T)=\log T\displaystyle\sum_{0<\gamma,\gamma'\leq T} &h_{a/\pi}\!\left((\gamma-\gamma')\frac{\log T}{2\pi}\right) w(\gamma-\gamma')-\dfrac{1}{2}\int_{1}^{T}\log^2\bigg(\dfrac{t}{2\pi}\bigg)\dt 
\\
&+ O\bigg(\dfrac{\log^4 T}{a^2}\bigg) +O(aT\log T),
\end{split}
\]
where $w(u)=4/(4+u^2)$.\footnote{The weight function $w(u)=2 h_2(u)$ is also a Poisson kernel, but we keep Montgomery's notation $w(u)$ to illustrate the connection to the Fourier inversion formula \eqref{inversion}. }
\end{lemma}
\begin{proof}
This formula is stated in \cite[Theorem 1]{GGM} without the weight function $w(\gamma-\gamma')$ in the double sum over zeros and with the constraint $0<a \ll 1$. The proof in \cite{GGM} goes through unchanged with the condition $0<a\le \sqrt{\log T}$ and a calculation on p.~115 of \cite[Section 2]{GGM} shows that the factor $w(\gamma-\gamma')$ can be added at the expense of a term that is $O(aT\log T)$. 
\end{proof}

\subsubsection{Extremal bandlimited approximations} Our argument for the upper bound for the second moment of the logarithmic derivative of $\zeta(s)$ is related to the following extremal problem in Fourier analysis.

\subsubsection*{Extremal problem 7 {\rm (EP7)}} Fix $b >0$ and let $h_b(x)$ be the Poisson kernel defined in \eqref{Def_Poisson_kernel_beta}. Find a continuous and integrable function $m_b:\R \to \R$ such that 
\begin{enumerate}
\item[(i)] $h_b(x) \leq m_b(x)$ for all $x \in \R$;

\smallskip

\item[(ii)] ${\rm supp}(\widehat{m_b}) \subset [-1,1];$

\smallskip

\item[(iii)] $\int_{\R}\big( m_b(x) - h_b(x) \big) \, \mathrm{d}x$ is as small as possible.
\end{enumerate}

\smallskip

This is called the Beurling-Selberg majorant problem (for the function $h_b$). As discussed in \cite[Lemma 9]{CChiM}, the solution of this particular problem comes from the general Gaussian subordination framework of Carneiro, Littmann, and Vaaler \cite{CLV}. Such extremal function exists and is unique, being given by
\[
m_{b}(x) = \left(\frac{b}{b^2 + x^2}\right) \left( \frac{e^{2\pi b} +  e^{-2\pi b} - 2\cos(2\pi x)}{\left(e^{\pi b}- e^{-\pi b}\right)^2} \right).
\]
Its Fourier transform is given by
\[
\widehat{m}_{b}(\alpha) = \dfrac{\pi}{2}\dfrac{\sinh(2\pi b(1-|\alpha|))}{(\sinh(\pi b))^2} \chi_{[-1,1]}(\alpha).
\]

\subsubsection{A weighted integral of $F(\alpha)$} For the lower bound in Theorem \ref{log zeta} we shall use a different approach rather than bandlimited approximations. Following Goldston \cite[Section 7]{G2}, we define the function
\[
\frak{I}(\xi)=\int_{1}^{\xi}(\xi-\alpha) F(\alpha) \, \mathrm{d} \alpha
\]
and we observe that $\frak{I}''(\xi)=F(\xi)$ for $\xi \ge 1$.  The following lemma gives a non-trivial lower bound for $\frak{I}(\xi)$ when $\xi \ge 1+1/\sqrt{3}$.

\begin{lemma} \label{CCLM_lemma}
Assume RH. Then, as $T\to\infty$, we have
\[
\frak{I}(\xi) \geqslant \frac{\xi^{2}}{2}-\xi+\frac{1}{3}+O\!\left(\xi^{2}  \sqrt{\frac{\log\log T}{\log T}} \, \right)
\]
uniformly for $\xi \ge 1$. 
\end{lemma}
\begin{proof}
This is a slight refinement of \cite[Lemma 17]{CCLM}, using \eqref{F formula} in the proof that appears there. 
\end{proof}

\subsection{Proof of Theorem \ref{log zeta}}
\subsubsection{Upper bound}
We use the special function $m_b(x)$ and Lemma \ref{GGM_Lemma}. Since $\widehat{m}_{b}(\alpha)$ and $F(\alpha)$ are even and $\mathrm{supp}(\widehat{m}_{b}) \subset [-1,1]$, by \eqref{inversion} and \eqref{F formula} we have 
\begin{align*}
\sum_{0<\gamma,\gamma'\leq T} &h_{a/\pi}\!\left((\gamma-\gamma')\frac{\log T}{2\pi}\right)  w(\gamma-\gamma') 
\\
&\qquad \le \sum_{0<\gamma,\gamma'\leq T} m_{a/\pi}\!\left((\gamma-\gamma')\frac{\log T}{2\pi}\right) w(\gamma-\gamma')
\\
&\qquad = \frac{T\log T}{2\pi} \int_{-1}^1 \widehat{m}_{a/\pi}(\alpha) \, F(\alpha) \, \mathrm{d}\alpha
\\
&\qquad = T\log T \left\{  \int_{0}^{1} \frac{\sinh (2a(1-\alpha))}{2\,(\sinh a)^2} \big(\alpha + T^{-2\alpha}\log T\big)\left(1+O\Bigg(\sqrt{\frac{\log\log T}{\log T}} \Bigg) \right)\d\alpha  \right\}
\\
& \qquad = T\log T \left\{ \dfrac{\coth a}{4a^2}-\dfrac{(\csch\, a)^2}{4a} + \dfrac{\coth a}{2} + O\Bigg(\bigg(\frac{1}{a}+1\bigg)\sqrt{\frac{\log\log T}{\log T}} \Bigg)  \right\},
\end{align*}
where the big-$O$ term is obtained by using the fact that $0<a \le \sqrt{\log T}$. Since
\[
\int_{1}^{T}\log^2\bigg(\dfrac{t}{2\pi}\bigg)\dt = T\log^2T + O\big(T\log T\big),
\]
%if we let
%\begin{equation}\label{U}
%U(a)=\dfrac{\coth a}{4a^2}-\dfrac{(\csch\, a)^2}{4a} + \dfrac{\coth a}{2}-\dfrac{1}{2},
%\end{equation}
the upper bound in Theorem \ref{log zeta} now follows from Lemma \ref{GGM_Lemma} by using the additional constraints
\begin{equation}\label{constraints}
\frac{(\log T)^{5/2}}{T}\leq a \le \frac{(\log T)^{1/4}}{(\log\log T)^{1/2}}
\end{equation}
and the fact that $U^+(a) = \frac{2}{3a} - \frac{1}{2}+O(a)$ as $a \to 0^+$ and $U^+(a) \sim \frac{1}{4a^2}$ as $a \to \infty$, in order to group the error terms.

\subsubsection{Lower bound} We now use Lemmas \ref{GGM_Lemma} and \ref{CCLM_lemma}. Since $\widehat{h}_b(\alpha) = \pi e^{-2 \pi b |\alpha|}$, by \eqref{inversion}, \eqref{F formula}, and the fact that $F(\alpha)$ is even, we have
\begin{align*}
\sum_{0<\gamma,\gamma'\leq T} &h_{a/\pi}\!\left((\gamma-\gamma')\frac{\log T}{2\pi}\right)  w(\gamma-\gamma')  = \frac{T\log T}{2\pi} \int_{-\infty}^\infty \widehat{h}_{a/\pi}(\alpha) \, F(\alpha) \, \mathrm{d}\alpha
\\
&\qquad  = T\log T \, \bigg\{  \int_{0}^{1} e^{-2 a \alpha}  \big(\alpha + T^{-2\alpha}\log T\big)\left(1+O\Bigg(\sqrt{\frac{\log\log T}{\log T}} \Bigg) \right)\d\alpha 
\\
&\qquad \qquad\qquad\qquad + \int_{1}^\infty e^{-2 a \alpha} \, F(\alpha) \, \mathrm{d}\alpha  \bigg\}
\\
&\qquad  = T\log T \, \bigg\{ \frac{1\!-\!(1\!+\!2 a) \, e^{-2 a}}{4 a^{2}}  +\frac{1}{2} + O\Bigg(\sqrt{\frac{\log\log T}{\log T}} \Bigg)
 + \int_{1}^\infty e^{-2 a \alpha} \, F(\alpha) \, \mathrm{d}\alpha  \bigg\},
%\\
%& \qquad = T\log T \left\{ \dfrac{\coth a}{4a^2}-\dfrac{(\csch\, a)^2}{4a} + \dfrac{\coth a}{2} + O\Bigg(\bigg(\frac{1}{a}+1\bigg)\frac{\sqrt{\log\log T}}{\sqrt{\log T}}\Bigg)  \right\},
\end{align*}
where the big-$O$ term is obtained by using the fact that $0<a \le \sqrt{\log T}$. To estimate the integral from 1 to $\infty$, we integrate by parts twice (from the work of Goldston \cite[Section 7]{G2} we have $\frak{I}'(\xi) = O(\xi)$ and $\frak{I}(\xi) = O(\xi^2)$ for $\xi \geq 1$). Since $\frak{I}(1)=\frak{I}'(1)=0$ and $\frak{I}(\alpha)\ge 0$ for $\alpha \ge 1$, we apply Lemma \ref{CCLM_lemma} to deduce that
\begin{align*}
\int_{1}^\infty e^{-2 a \alpha} \, F(\alpha) \, \mathrm{d}\alpha &= 4 a^{2} \int_{1}^{\infty} \frak{I}(\alpha) \, e^{-2 a \alpha} \, \mathrm{d} \alpha
\\
&\ge 4a^2 \int_{1+1 / \sqrt{3}}^{\infty}\left(\frac{\alpha^{2}}{2}-\alpha+\frac{1}{3}\right) e^{-2 a \alpha} \, \mathrm{d} \alpha +O\left(a^2 \sqrt{\frac{\log\log T}{\log T}} \int_{1+1 / \sqrt{3}}^{\infty} \alpha^{2} \, e^{-2 a \alpha} \, \mathrm{d} \alpha\right)
\\
&=\left(\frac{1}{2 a}+\frac{1}{\sqrt{3}}\right) e^{-2 a(1+1 / \sqrt{3})} + O\left(\frac{1}{a}\,\sqrt{\frac{\log\log T}{\log T}} \right).
\end{align*}
%{\color{red} Do we have an issue with the uniformity in $a$ at this step?} 
Again since
\[
\int_{1}^{T}\log^2\bigg(\dfrac{t}{2\pi}\bigg)\dt = T\log^2T + O\big(T\log T\big),
\]
%if we let
%\begin{equation}\label{U}
%U(a)=\dfrac{\coth a}{4a^2}-\dfrac{(\csch\, a)^2}{4a} + \dfrac{\coth a}{2}-\dfrac{1}{2},
%\end{equation}
the lower bound in Theorem \ref{log zeta} now follows from Lemma \ref{GGM_Lemma} by using the additional constraints in \eqref{constraints}
%\[
%\frac{(\log T)^{5/2}}{T}\leq a \le \frac{(\log T)^{1/4}}{(\log\log T)^{1/2}}
%\]
%{\color{red} is this not true anymore?...} 
and the fact that $U^-(a) = \frac{1}{2a} - \frac{1}{2}+O(a^2)$ as $a \to 0^+$ and $U^-(a) \sim \frac{1}{4a^2}$ as $a \to \infty$, in order to group the error terms. This concludes the proof.

\section{Appendix A: Minima of Dirichlet kernels}\label{App_A}
Complementing the discussion in \S \ref{Subsec_20201221_09:26}, we present a brief proof of inequality \eqref{20210110_11:41}. Let $\mathfrak{m}(n)$ as in \eqref{20200915_13:33} and $c_0$ as in \eqref{20201221_09:30}.
\begin{proposition}
For each $n \in \N$, the following bounds hold
\begin{equation*}	
	2c_0  - \frac{(2 \pi -1)}{n}  \ \leq \ \frac{\mathfrak{m}(n)}{n} \ \leq \ 2c_0 + \frac{5.4935}{n}.
\end{equation*}
\end{proposition}
\begin{proof}
We rewrite \eqref{9_14_8:00pm} as
	$$
	D_n(n)=1+2n\displaystyle\sum_{k=1}^{n}
	\cos\bigg(xn\dfrac{k}{n}\bigg)\dfrac{1}{n}.$$
	Using the mean value theorem we get, for $x\geq 0$, 
	\begin{align*}
	 \Bigg|\displaystyle\sum_{k=1}^{n}& \cos\bigg(xn\dfrac{k}{n}\bigg)\dfrac{1}{n}-\displaystyle\int_{0}^{1}\cos(xnt)\,\dt\,\Bigg|  =\Bigg|\displaystyle\sum_{k=1}^{n}\cos\bigg(xn\dfrac{k}{n}\bigg)\dfrac{1}{n}-\displaystyle\sum_{k=1}^{n}\displaystyle\int_{(k-1)/n}^{k/n}\cos(xn t)\,\dt\,\Bigg|\\
	& \leq \displaystyle\sum_{k=1}^{n}\Bigg|\displaystyle\int_{(k-1)/n}^{k/n}\bigg(\cos\bigg(xn\dfrac{k}{n}\bigg)-\cos(xn t)\bigg)\,\dt\,\Bigg|\\
	& \leq \displaystyle\sum_{k=1}^{n}\displaystyle\int_{(k-1)/n}^{k/n}xn \left(\frac{k}{n} - t \right)\dt\\
	& =  \frac{x}{2}.
	\end{align*}
Therefore, 
	\begin{align} \label{9_15_3:50pm}
	\dfrac{1}{n}+\dfrac{2\sin(nx)}{nx}-x\leq \dfrac{D_n(x)}{n}\leq \dfrac{1}{n}+\dfrac{2\sin(nx)}{nx}+x.
	\end{align}
Let $x_1 = 4.49340\ldots$ be the unique real positive number such that 
$$c_0 = \min_{x\in\mathbb{R}} \frac{\sin x}{x} = \frac{\sin x_1}{x_1} =  -0.21723\ldots\!.$$
Plugging $x_n=x_1/n$ in \eqref{9_15_3:50pm} we obtain
	\begin{align*}
	\dfrac{\mathfrak{m}(n)}{n} & \leq \dfrac{D_n(x_n)}{n}\leq \dfrac{1}{n}+\dfrac{2\sin(nx_n)}{nx_n}+x_n = \dfrac{1}{n}+\dfrac{2\sin(x_1)}{x_1}+\dfrac{x_1}{n} \leq 2c_0 + \dfrac{5.4935}{n}.
	\end{align*}
On the other hand, using the fact that $D_n(x)$ is an even periodic function with period $2\pi$, it follows that $\mathfrak{m}(n)  = \displaystyle\min_{x\in [0,\pi]}D_n(x)$. Let $\xi\in [0,\pi]$ be a real number where such minimum is attained. If $2\pi/(2n+1)\leq \xi\leq 4\pi/(2n+1)$, using \eqref{9_15_3:50pm} we get
\begin{align*}
\dfrac{\mathfrak{m}(n)}{n}  =\dfrac{D_n(\xi)}{n}\geq \dfrac{1}{n}+2c_0-\dfrac{4\pi}{2n+1}> 2c_0 -  \dfrac{(2\pi-1)}{n}. 
\end{align*} 
If $6\pi/(2n+1)\leq \xi\leq \pi$, using the fact that $\sin t\geq 2t/{\pi}$ for $t\in\big[0,\frac{\pi}{2}\big]$ we have
\begin{align*} 
\dfrac{\mathfrak{m}(n)}{n} = \dfrac{D_n(\xi)}{n}& = \frac{\sin((n+1/2)\xi)}{n\sin(\xi/2)} \geq \frac{-1}{n\sin(\xi/2)}\geq -\dfrac{\pi}{n\xi}\geq -\dfrac{2n+1}{6n}>2c_0 -  \dfrac{(2\pi-1)}{n}. 
\end{align*}
Finally, in the cases $0\leq \xi < 2\pi/(2n+1)$ or $4\pi/(2n+1)< \xi < 6\pi/(2n+1)$, it is clear that $D_n(\xi)\geq 0$, and such points will not be points where the global minimum is attained. This concludes the proof.
\end{proof}

\section{Appendix B: Hilbert spaces and pair correlation}

\subsection{Sharp equivalence of norms} We conclude by revisiting a result of \cite{CCLM}, a paper that provides a study of the pair correlation of zeros of zeta via the framework of Hilbert spaces of entire functions. Let us first recall some basic terminology. For $\Delta >0$ we say that an entire function $f: \mathbb C \rightarrow \mathbb C$ has exponential type at most $2\pi\Delta$ if, for all $\varepsilon > 0$, there exists a positive constant $C_\varepsilon$ such that $|f(z)| \leq C_{\varepsilon}\,e^{(2\pi\Delta + \varepsilon)|z|}$ for all $z \in \C$. Let $\mc{B}_2(\pi\Delta)$ be the classical Paley--Wiener space, i.e.~the Hilbert space of entire functions of exponential type at most $\pi \Delta$ with norm
\begin{equation*}
\|f\|_{2} = \left(\int_{-\infty}^\infty |f(x)|^2 \, \dx\right)^{1/2} <\infty.
\end{equation*}
Functions in $\mc{B}_2(\pi\Delta)$ have Fourier transforms supported in the interval $[-\Delta/2, \Delta/2]$ (by the Paley--Wiener theorem). For a survey on such spaces, their interpolation formulas and some classical applications to analytic number theory we refer the reader to the work of J.~D.~Vaaler \cite{V}.

\smallskip

Write $\dmu(x) =  \big\{ 1 - \left(\frac{\sin \pi x}{\pi x}\right)^2\!\big\} \,\dx$ for the pair correlation measure and denote by $\mc{B}_2(\pi,\mu)$ the normed vector space of entire functions $f$ of exponential type at most $\pi$ with norm
\begin{equation*}
\|f\|_{L^2(\d\mu)} = \left(\int_{-\infty}^\infty |f(x)|^2 \, \dmu(x)\right)^{1/2} <\infty.
\end{equation*}
Using the uncertainty principle for the Fourier transform, it was shown in \cite[Lemma 12]{CCLM} that the vector spaces $\mc{B}_2(\pi)$ and $\mc{B}_2(\pi,\mu)$ are the same (as sets), with the norms being equivalent. That is, there is a universal constant ${\bf D} >0$ such that 
\begin{equation}\label{20210113_11:22}
{\bf D}\, \|f\|_2 \leq \|f\|_{L^2(\d\mu)} \leq  \|f\|_2
\end{equation}
whenever $f \in \mc{B}_2(\pi)$. In particular, $\mc{B}_2(\pi,\mu)$ is also a Hilbert space. It should be clear that the inequality on the right-hand side of \eqref{20210113_11:22} is sharp and that there are no extremizers. In fact, given any $f_0 \in \mc{B}_2(\pi)$, the sequence $f_n(z):=f_0(z-n)$ is an extremizing sequence as $n \to \infty$. In this appendix we discuss the problem of finding the value of the sharp constant ${\bf D}$.
\subsubsection*{Extremal problem 8 {\rm (EP8)}} Find 
\begin{align}\label{20210113_14:22}
{\bf D}^2 :=  \inf_{\substack{f \in \mc{B}_2(\pi) \\ f\neq 0}} \ \frac{\|f\|_{L^2(\d\mu)}^2 }{\|f\|_2^2} = \inf_{\substack{g \in \mathcal{A}_0 \\ g\neq 0}} \ \frac{\rho(g) - g(0)}{\widehat{g}(0)}.
\end{align}
{\sc Remark:} We comment briefly on the equality between the infima above, as it relates to the class $\mathcal{A}_0$ defined in \S \ref{Fourier_Opt}, the quantity $\rho(g)$ defined in \eqref{20201217_11:41} (which is equal to \eqref{20210205_14:19} in this case), and some of the other extremal problems that have been considered in this paper. This is essentially a consequence of the Paley--Wiener theorem and Krein's decomposition \cite[p.~154]{A}: a continuous and non-negative function $g \in L^1(\R)$ has ${\rm supp}(\widehat{g}) \subset [-1,1]$ if and only if it is the restriction to $\R$ of an entire function of exponential type $2\pi$ (that we keep calling $g(z)$) and $g(z) = f(z)\overline{f(\overline{z})}$ for some $f \in \mc{B}_2(\pi)$. The fact that we can restrict the search on the right-hand side of \eqref{20210113_14:22} to even functions comes from a standard symmetrization procedure: if $g$ is not even, we can consider $h(x) = \frac12(g(x) + g(-x))$ without affecting the ratio. 

\smallskip

Finding the sharp forms of embeddings between function spaces is usually a rich and non-trivial problem in analysis. As we shall see, extremal problem (EP8) has a particularly intriguing answer.

\begin{theorem}\label{Thm_20201218_14:43}
We have
$$ {\bf D}^2  = 1 - \frac{1}{2 \pi^2 \theta^2} = 0.3244\ldots\,,$$ 
where $0 < \theta < \frac{1}{2}$ is the unique solution of 
$$ \ \ \ \  \ \ \ \ \ \qquad\qquad \ \   \ \ \left( \pi \theta \right)\, \tan\left( \pi \theta \right) = 1 \ \ \ \ ;  \ \ \ \ (\theta = 0.27385\ldots).$$
Moreover, there is a unique $($up to multiplication by a non-zero complex constant\,$)$ extremal function $f \in \mc{B}_2(\pi)$ such that $\|f\|_{L^2(\d\mu)} / \|f\|_2 = {\bf D}$, namely
\begin{equation}\label{20210113_17:11}
f(z) = \frac{\sin \pi(z + \theta)}{\pi (z + \theta)} + \frac{\sin \pi(z - \theta)}{\pi (z - \theta)}.
\end{equation}
\end{theorem}

We shall give two different proofs of this result. The first one relies on classical interpolation formulas and we work mostly on the entire function side. The second one uses a variational approach and we work mostly on the Fourier transform side.

\subsection{Proof 1: Interpolation approach}

\subsubsection{Setup and uniqueness of extremizer} \label{App_sub1}If $h \in \mc{B}_2(\pi\Delta)$, from the Paley--Wiener theorem and Plancherel's identity, the norm $\|h\|_2$ can be inferred from $(1/\Delta)$-equally spaced samples as follows: for any $x_0 \in \R$,
\begin{equation}\label{20210113_15:38}
\|h\|_2^2 = \frac{1}{\Delta}\sum_{k \in \Z} \big|h\big(x_0 + \tfrac{k}{\Delta}\big)\big|^2.
\end{equation}
Let $f \in \mc{B}_2(\pi)$, and assume without loss of generality that $\|f\|_2 = 1$. Then, from \eqref{20210113_15:38}, 
 \begin{equation}\label{20210116_22:11}
1 = \|f\|_2^2  = \sum_{n \in \Z} \big|f\big(n - \tfrac{1}{2} \big)\big|^2.
\end{equation}
In fact, given any square summable sequence $\big\{f\big(n - \frac12\big)\big\}_{n \in \Z}$, this completely determines the function $f$ via the interpolation formula
 \begin{equation}\label{20210125_15:59}
 f(z) = \sum_{n \in \Z}f\big(n - \tfrac{1}{2} \big) \frac{\sin \pi \big(z  -n+ \tfrac12 \big)}{\pi \big(z -n + \tfrac12 \big)}.
 \end{equation}
 In particular, we have
 \begin{equation}\label{20210121_11:57}
 f(0) = \sum_{n \in \Z}f\big(n - \tfrac{1}{2} \big) \frac{\sin \pi \big(-n+ \tfrac12 \big)}{\pi \big(-n+ \tfrac12 \big)} = \sum_{n \in \Z}f\big(n - \tfrac{1}{2} \big) \frac{(-1)^{n+1}}{\pi \big(n- \tfrac12 \big)}.
 \end{equation}
 Since $f(z)\left(\frac{\sin \pi z}{\pi z}\right) \in \mc{B}_2(2\pi)$, we may apply \eqref{20210113_15:38} to get
 \begin{align}\label{20210122_10:04}
\begin{split}
\|f\|_{L^2(\d\mu)}^2  & = \int_{-\infty}^\infty |f(x)|^2\, \dx - \int_{-\infty}^\infty |f(x)|^2\,\left(\frac{\sin \pi x}{\pi x}\right)^2 \,\dx\\
& = 1 - \frac{1}{2}\sum_{k \in \Z} \big|f\big( \tfrac{k}{2}\big)\big|^2\,\left(\frac{\sin \pi \big(\tfrac{k}{2}\big)}{\pi \big(\tfrac{k}{2}\big)}\right)^2 \\
& = 1 -  \frac{1}{2}|f(0)|^2 - \frac{1}{2\pi^2}\sum_{n \in \Z} \frac{\big|f\big(n - \tfrac{1}{2} \big)\big|^2}{ \big|n - \tfrac{1}{2} \big|^2} .
\end{split}
\end{align}
The problem then becomes: given a square summable sequence of coefficients $\big\{f\big(n - \frac12\big)\big\}_{n \in \Z}$, normalized as in \eqref{20210116_22:11}, we seek to maximize the quantity
\begin{equation*}
\frac{1}{2}|f(0)|^2 + \frac{1}{2\pi^2}\sum_{n \in \Z} \frac{\big|f\big(n - \tfrac{1}{2} \big)\big|^2}{ \big|n - \tfrac{1}{2} \big|^2} .
\end{equation*}
By the triangle inequality, from \eqref{20210121_11:57} we have
 \begin{equation}\label{20210116_22:42}
 |f(0)| \leq \sum_{n \in \Z} \frac{\big|f\big(n - \tfrac{1}{2} \big)\big|}{\pi \big|n - \tfrac12\big|}.
 \end{equation}
Equality holds in \eqref{20210116_22:42} if and only if there is a complex number $c$, with $|c| =1$, such that 
 \begin{equation}\label{20210121_15:57}
  \frac{f\big(n - \tfrac{1}{2} \big) (-1)^{n+1}}{\big(n -  \tfrac12 \big)} = c  \ \frac{\big|f\big(n - \tfrac{1}{2} \big)\big|}{ \big|n - \tfrac12\big|}
  \end{equation}
for all $n \in \Z$, and we may henceforth assume that this is the case. It is then enough to decide what is the best sequence of absolute values $\big\{\big|f\big(n - \frac12\big)\big|\big\}_{n \in \Z}$, in order to maximize 
\begin{equation}\label{20210121_12:26}
\frac{1}{2}|f(0)|^2 + \frac{1}{2\pi^2}\sum_{n \in \Z} \frac{\big|f\big(n - \tfrac{1}{2} \big)\big|^2}{ \big(n - \tfrac{1}{2} \big)^2} =\frac{1}{2\pi^2}\left(  \left(\sum_{n \in \Z} \frac{\big|f\big(n - \tfrac{1}{2} \big)\big|}{ \big|n - \tfrac12\big|}\right)^2 + \sum_{n \in \Z} \frac{\big|f\big(n - \tfrac{1}{2} \big)\big|^2}{ \big|n - \tfrac{1}{2} \big|^2}\right).
\end{equation}
Note that if we symmetrize the sequence by considering 
$$\big|\widetilde{f}\big(n - \tfrac{1}{2}\big)\big|:= \left( \frac{\big|f\big(n - \tfrac{1}{2}\big)\big|^2 + \big|f\big(-n + \tfrac{1}{2}\big)\big|^2}{2}\right)^{1/2}\,,$$
we get an outcome at least as good in \eqref{20210121_12:26}, doing strictly better if there is $n \in \Z$ for which $\big|f\big(n - \tfrac{1}{2} \big)\big| \neq \big|f\big(-n + \tfrac{1}{2} \big)\big|$ (observe that the second sum on the right-hand side of \eqref{20210121_12:26} remains unchanged while the first one, before raising to the power $2$, does not decrease by an application of the inequality $a + b \leq 2((a^2 + b^2)/2)^{1/2}$ to the symmetric pairs). Therefore, from now on we can also assume that 
\begin{equation}\label{20210121_15:04}
\big|f\big(n - \tfrac{1}{2} \big)\big| = \big|f\big(-n + \tfrac{1}{2} \big)\big|
\end{equation}
for all $n \in \Z$. By a similar argument, we can also make the following remark: if two sequences $\big\{\big|f_1\big(n - \frac12\big)\big|\big\}_{n \in \Z}$ and $\big\{\big|f_2\big(n - \frac12\big)\big|\big\}_{n \in \Z}$, verifying \eqref{20210116_22:11} and \eqref{20210121_15:04}, yield the same value in \eqref{20210121_12:26}, we can construct a third sequence $\big\{\big|f_3\big(n - \frac12\big)\big|\big\}_{n \in \Z}$ by
$$\big|f_3\big(n - \tfrac{1}{2}\big)\big|:= \left( \frac{\big|f_1\big(n - \tfrac{1}{2}\big)\big|^2 + \big|f_2\big(n - \tfrac{1}{2}\big)\big|^2}{2}\right)^{1/2}\,,$$
and get an outcome at least as good in \eqref{20210121_12:26}. This follows since
$$\sum_{n \in \Z} \frac{\big|f_3\big(n - \tfrac{1}{2} \big)\big|^2}{ \big|n - \tfrac{1}{2} \big|^2} = \frac{1}{2} \left(\sum_{n \in \Z} \frac{\big|f_1\big(n - \tfrac{1}{2} \big)\big|^2}{ \big|n - \tfrac{1}{2} \big|^2} + \sum_{n \in \Z} \frac{\big|f_2\big(n - \tfrac{1}{2} \big)\big|^2}{ \big|n - \tfrac{1}{2} \big|^2}\right),$$
and
\begin{align*}
\left(\sum_{n \in \Z} \frac{\big|f_3\big(n - \tfrac{1}{2} \big)\big|}{ \big|n - \tfrac12\big|}\right)^2 & = \sum_{n \in \Z} \frac{\big|f_3\big(n - \tfrac{1}{2} \big)\big|^2}{ \big|n - \tfrac{1}{2} \big|^2} +  \sum_{m\neq n } \frac{\big|f_3\big(n - \tfrac{1}{2} \big)\big|\ \big|f_3\big(m - \tfrac{1}{2} \big)\big|}{ \big|n - \tfrac{1}{2} \big|\  \big|m - \tfrac{1}{2} \big|} \\
& \geq  \sum_{n \in \Z} \frac{\big|f_3\big(n - \tfrac{1}{2} \big)\big|^2}{ \big|n - \tfrac{1}{2} \big|^2} +  \frac{1}{2}\left( \sum_{m\neq n } \frac{\big|f_1\big(n - \tfrac{1}{2} \big)\big|\ \big|f_1\big(m - \tfrac{1}{2} \big)\big|}{ \big|n - \tfrac{1}{2} \big|\  \big|m - \tfrac{1}{2} \big|} +  \sum_{m\neq n } \frac{\big|f_2\big(n - \tfrac{1}{2} \big)\big|\ \big|f_2\big(m - \tfrac{1}{2} \big)\big|}{ \big|n - \tfrac{1}{2} \big|\  \big|m - \tfrac{1}{2} \big|}\right)\\
& = \frac{1}{2}\left( \left(\sum_{n \in \Z} \frac{\big|f_1\big(n - \tfrac{1}{2} \big)\big|}{ \big|n - \tfrac12\big|}\right)^2 + \left(\sum_{n \in \Z} \frac{\big|f_2\big(n - \tfrac{1}{2} \big)\big|}{ \big|n - \tfrac12\big|}\right)^2\right).
\end{align*}
Note above the use of the inequality $(ab + cd) \leq (a^2 + c^2)^{1/2}\, (b^2 + d^2)^{1/2}$. Equality of the outcome in \eqref{20210121_12:26} happens if and only if 
$$\big|f_1\big(n - \tfrac{1}{2} \big)\big|\ \big|f_2\big(m - \tfrac{1}{2} \big)\big|= \big|f_2\big(n - \tfrac{1}{2} \big)\big|\ \big|f_1\big(m - \tfrac{1}{2} \big)\big|$$
for all $m,n \in \Z$. This means that the two sequences would have to be proportional and the normalization \eqref{20210116_22:11} would force them to be equal. This discussion leads us to the following conclusion: once we prove that a maximizer exists, it is going to be unique (modulo multiplication by a constant $c$ as in \eqref{20210121_15:57}).

\medskip

Under the symmetry condition \eqref{20210121_15:04} we may rewrite the quantity on the right-hand side of \eqref{20210121_12:26} as 
\begin{equation}\label{20210122_10:06}
\frac{4}{\pi^2}\left( 2 \left(\sum_{n=1}^{\infty} \frac{\big|f\big(n - \tfrac{1}{2} \big)\big|}{ 2n - 1}\right)^2 + \sum_{n = 1}^{\infty} \frac{\big|f\big(n - \tfrac{1}{2} \big)\big|^2}{ (2n -1)^2}\right).
\end{equation}
This needs to be maximized under the constraint
$$\sum_{n = 1}^{\infty} \big|f\big(n - \tfrac{1}{2} \big)\big|^2 = \frac{1}{2}.$$

\subsubsection{An equivalent inequality and Lagrange multipliers} We now consider a reformulation of our problem, which is the case $N = \infty$ below. It is convenient for our purposes to also consider a finite-dimensional formulation in the  upcoming discussion.

\subsubsection*{Extremal problem 9 {\rm (EP9)}} Let $N \in \N$ or $N = \infty$. Let $\{a_n\}_{n =1}^{N}$ a sequence of non-negative real numbers, such that 
\begin{equation*}
\sum_{n =1}^{N} a_n^2 =\frac{1}{2}.
\end{equation*}
Find the supremum $Q_{N}$ of 
$$ F\big(\{a_n\}\big) = 2 \left( \sum_{n =1}^{N} \frac{a_n}{2n-1}\right)^2 + \sum_{n =1}^{N} \frac{a_n^2}{(2n-1)^2}.$$ 

\smallskip

One can verify that $\{Q_N\}_{N \in \N}$ is an increasing sequence and that 
\begin{equation}\label{20210122_09:12}
\lim_{N \to \infty} Q_N  = Q_{\infty}.
\end{equation}
Let us then focus in understanding the problem at a fixed level $N < \infty$. Observe first that we can assume without loss of generality that our sequence is ordered as $a_1 \geq a_2 \geq a_3 \geq \ldots \geq a_N$. This is a consequence of the rearrangement inequality for sequences. Also, since $0 \leq a_n \leq \frac{\sqrt{2}}{2}$ for $1 \leq n \leq N$ and our functional $F$ is continuous, the maximum $Q_N$ is attained by some sequence. Our domain, in principle, is the cube $\big[0,\frac{\sqrt{2}}{2}\big]^N$, and we are restricted by the constraint function:
\begin{equation}\label{20210117_00:50}
G(a_1, a_2, \ldots, a_N) = \sum_{n =1}^{N} a_n^2 =\frac{1}{2}.
\end{equation}
We first study the critical points in the interior of our domain. Assume that $(a_1, a_2, a_3, \ldots, a_N)$ with 
  $$a_1 \geq a_2 \geq a_3 \geq \ldots \geq a_N >0$$
is a critical point of our functional (note that the assumption that $a_N>0$ automatically implies that $\frac{\sqrt{2}}{2} >a_1$ and all these points are in the open interval $\big(0,\frac{\sqrt{2}}{2}\big)$). By the Lagrange multiplier theorem, we must have
$$\frac{\partial F}{\partial a_k}  = \lambda \frac{\partial G}{\partial a_k}$$
for $k =1,2,\ldots, N$ and some $\lambda \in \R$. This gives us the following system of equations:
\begin{equation}\label{20210117_01:10}
4 \left( \sum_{n =1}^{N} \frac{a_n}{2n-1}\right)\frac{1}{(2k-1)} + 2\frac{a_k}{(2k-1)^2} = 2\lambda a_k \ \ \ \ ;  \ \ \ (k =1,2,\ldots, N).
\end{equation}
From \eqref{20210117_01:10} with $k=1$, we plainly see that $\lambda >1$. We may rewrite \eqref{20210117_01:10} as
\begin{equation}\label{20210117_00:32}
2 \left( \sum_{n =1}^{N} \frac{a_n}{2n-1}\right)  = a_k\left( (2k-1)\lambda - \frac{1}{(2k-1)}\right) \ \ \ \  ; \ \ \ (k =1,2,\ldots, N).
\end{equation}
Since the quantity on the left-hand side of \eqref{20210117_00:32} is fixed, we must have
\begin{equation}\label{20202117_00:38}
a_k = a_1 \frac{\lambda -1}{\left( (2k-1)\lambda - \frac{1}{(2k-1)}\right)} \ \ ;\ \ (k =1,2,\ldots, N).
\end{equation}
We can now use \eqref{20210117_00:32} and \eqref{20202117_00:38} to find the value of $\lambda$, getting 
\begin{equation}\label{20210117_01:30}
 \left( \sum_{n =1}^{N} \frac{1}{ (2n-1)^2\lambda - 1}\right)  = \frac{1}{2}.
 \end{equation}
This uniquely determines our $\lambda=: \lambda_N$ (in case we need to highlight the dependence on the parameter $N$, we shall use $\lambda_N$). Once $\lambda = \lambda_N$ is found, the value of $a_1$ can be computed by the constraint function \eqref{20210117_00:50} and \eqref{20202117_00:38}, giving
\begin{equation}\label{20210117_02:01}
a_1^2 \sum_{n=1}^N \frac{(\lambda -1)^2}{\left( (2n-1)\lambda - \frac{1}{(2n-1)}\right)^2} = \frac{1}{2}.
\end{equation}
From \eqref{20202117_00:38} we have the other $a_k$'s. Hence this interior critical point is unique. We can also find the explicit value it yields in the outcome functional $F$. Dividing \eqref{20210117_01:10} by $2$, multiplying by $a_k$, and adding up over $k$ from $1$ to $N$, we get 
\begin{align}\label{20210118_09:19}
F(a_1, a_2, \ldots, a_N) = \frac{\lambda_N}{2}.
\end{align}
Alternatively, we can obtain \eqref{20210118_09:19} directly from \eqref{20202117_00:38}, \eqref{20210117_01:30}, and \eqref{20210117_02:01}.

\subsubsection{Conclusion} It is clear from \eqref{20210117_01:30} that 
$$\lambda_1 < \lambda_2  < \ldots < \lambda_N < \lambda_{N+1} < \ldots$$
This is important for us for the following reason. Suppose that we are solving the problem at level $N < \infty$ and we have a global maximizer with $a_1 \geq a_2 \geq a_3 \geq \ldots \geq a_N \geq 0$. Suppose that $a_{M+1} = \ldots  = a_N =0$ and that $a_M >0$. Then $(a_1, a_2, \ldots , a_M)$ must be a global maximizer at level $M$ (for if we had a sequence doing better, we would just add some zeros and do better at level $N$ as well). Hence, this sequence $(a_1, a_2, \ldots , a_M)$ at level $M$ is an interior critical point, which we have seen is unique and yields the value
$$F(a_1, a_2, \ldots, a_M) = \frac{\lambda_M}{2}.$$
On the other hand, we know that the interior critical point at level $N$ gives the value $\lambda_N/2$, which is strictly bigger than $\lambda_M/2$, a contradiction. Hence, a global maximizer at level $N$ must have $a_N >0$, and it will be the unique interior critical point constructed with the Lagrange multipliers. The conclusion is that 
$$Q_N = \frac{\lambda_N}{2}.$$

\smallskip

Note that $\displaystyle\lim_{N \to \infty} \lambda_N = \lambda_{\infty}\,,$
where $\lambda = \lambda_{\infty} >1$ solves the equation (see \cite[Eq.~1.421-1]{GR})
\begin{equation}\label{20210117_01:39}
 \ \ \ \ \ \ \  \frac{1}{2} = \left( \sum_{n =1}^{\infty} \frac{1}{ (2n-1)^2\lambda - 1}\right)  =  \left( \frac{\pi}{4 \sqrt{\lambda}}\right)\  \tan\left( \frac{\pi}{2 \sqrt{\lambda}}\right) \ \ \ ; \ \ \ (\lambda_{\infty} = 3.33354\ldots).
 \end{equation}
The final answer of extremal problem (EP9) in the case $N = \infty$ is then given by \eqref{20210122_09:12}, namely
\begin{equation}\label{20210117_01:57}
Q_{\infty} = \frac{\lambda_{\infty}}{2}.
\end{equation}
We observe that there exists a maximizing sequence. This is given by \eqref{20202117_00:38} and \eqref{20210117_02:01} for this particular choice of $\lambda = \lambda_{\infty}$ (taking $N = \infty$ in these identities). The ideas in the discussion at the end of \S \ref{App_sub1} show that such maximizing sequence is unique for problem (EP9) when $N = \infty$, as well.

\smallskip

As for our original problem, from \eqref{20210122_10:04}, \eqref{20210121_12:26}, \eqref{20210122_10:06}, and \eqref{20210117_01:57}, we conclude that 
\begin{equation*}
\|f\|_{L^2(\d\mu)}^2  \geq 1 - \frac{2 \lambda_{\infty} }{\pi^2} = 0.3244\ldots
\end{equation*}
Equality can be attained and the unique maximizer is given by \eqref{20210121_15:57}, \eqref{20210121_15:04}, \eqref{20202117_00:38}, and \eqref{20210117_02:01} (with $N = \infty$ in the last two), yielding
\begin{equation}\label{20210125_15:49}
f\big(-n + \tfrac{1}{2} \big) = f\big(n - \tfrac{1}{2} \big)  = (-1)^{n+1}\, \frac{\sqrt{2}}{2} \frac{1}{\left( (2n-1)\lambda_{\infty} - \frac{1}{(2n-1)}\right)}\left(\sum_{k=1}^{\infty}  \frac{1}{\left( (2k-1)\lambda_{\infty} - \frac{1}{(2k-1)}\right)^2}\right)^{-1/2}.
\end{equation}
We remark that, for any $\lambda >1$, we have the identity (that follows from \eqref{20210117_01:39} by differentiation)
\begin{align}\label{20210122_11:24}
\sum_{k=1}^{\infty}  \frac{1}{\left( (2k-1)\lambda - \frac{1}{(2k-1)}\right)^2} = \frac{\pi^2 \sec^2\left( \frac{\pi}{2 \sqrt{\lambda}}\right) + 2 \pi \sqrt{\lambda }\tan\left( \frac{\pi}{2 \sqrt{\lambda}}\right)}{16\lambda^2}. 
\end{align}
Hence, when $\lambda = \lambda_{\infty}$, we can use \eqref{20210117_01:39} to simplify \eqref{20210122_11:24} to $\frac{\pi^2}{16 \lambda_{\infty}^2} + \frac{1}{2 \lambda_{\infty}}$.

\medskip 

With the substitution $\theta = \frac{1}{2\sqrt{ \lambda_{\infty}}}$, we may check directly that the function presented in \eqref{20210113_17:11}, when evaluated at $\Z + \frac{1}{2}$, gives something proportional to \eqref{20210125_15:49}. The interpolation formula \eqref{20210125_15:59} then guarantees that \eqref{20210113_17:11} is indeed (a multiple of) our maximizer. This concludes the proof.

\subsection{Proof 2: Variational approach}

\subsubsection{Existence of extremizers} The first step in this approach is to show that there exists $f \in \mc{B}_2(\pi)$ that extremizes \eqref{20210113_14:22} (i.e.~such that $\|f\|_{L^2(\d\mu)} / \|f\|_2 = {\bf D}$). As we have argued in \eqref{20210113_14:22} and the remark thereafter, it is enough to find an extremizer in the class $\mathcal{A}_0$ defined in \S \ref{Fourier_Opt} for 
\begin{equation}\label{20210122_18:49}
1 - {\bf D}^2:= \sup_{\substack{ g \in \mathcal{A}_0 \\ g\neq 0}}\frac{\int_{-\infty}^{\infty} g(x) \left( \frac{\sin \pi x}{\pi x}\right)^2\,\dx}{\int_{-\infty}^{\infty} g(x) \,\dx}.
\end{equation}
Let $\{g_n\}_{n \geq 1} \subset \mathcal{A}_0$ be an extremizing sequence for \eqref{20210122_18:49}, normalized so that $\|g_n\|_1 = 1$ for all $n$. Hence, 
$$\int_{-\infty}^{\infty} g_n(x) \left( \frac{\sin \pi x}{\pi x}\right)^2\,\dx \to 1 - {\bf D}^2$$
as $n \to \infty$. Recall that ${\rm \supp}(\widehat{g_n}) \subset [-1,1]$ and that $\|\widehat{g_n}\|_{\infty} = \widehat{g_n}(0) = \|g_n\|_1 = 1$. Therefore $\|\widehat{g_n}\|_{2}^2 \leq 2\|\widehat{g_n}\|_{\infty}^2 \leq 2$, and we see that $\{g_n\}_{n \geq 1}$ is a bounded sequence in $\mc{B}_2(2\pi)$. By reflexivity, passing to a subsequence if necessary, we may assume that $g_n$ converges weakly to a certain $g^{\sharp} \in \mc{B}_2(2\pi)$. In particular,
\begin{align}\label{20210125_09:22}
1 - {\bf D}^2 = \lim_{n \to \infty} \int_{-\infty}^{\infty} g_n(x) \left( \frac{\sin \pi x}{\pi x}\right)^2\,\dx = \int_{-\infty}^{\infty} g^{\sharp} (x) \left( \frac{\sin \pi x}{\pi x}\right)^2\,\dx,
\end{align}
and hence $g^{\sharp}  \neq 0$. Since $\mc{B}_2(2\pi)$ is a reproducing kernel Hilbert space, we also have the pointwise convergence
\begin{align*}
 \lim_{n \to \infty} g_n(y) =  \lim_{n \to \infty} \int_{-\infty}^{\infty} g_n(x) \, \frac{\sin 2\pi(y-x)}{\pi (y-x)}  \,\dx = \int_{-\infty}^{\infty} g^{\sharp} (x) \, \frac{\sin 2\pi(y-x)}{\pi (y-x)}  \,\dx = g^{\sharp} (y)
\end{align*}
for all $y \in \R$. Hence $g^{\sharp} $ is even and non-negative on $\R$. Moreover, by Fatou's lemma, it follows that
\begin{equation}\label{20210125_09:23}
\|g^{\sharp} \|_1 \leq \liminf_{n \to \infty} \|g_n\|_1 =1 ,
\end{equation}
which implies that $g^{\sharp}  \in \mathcal{A}_0$. From \eqref{20210125_09:22} and \eqref{20210125_09:23}, we see that this particular $g^{\sharp} $ is an extremizer for \eqref{20210122_18:49}.

\subsubsection{Solving the Euler-Lagrange equation} For a generic $0\neq h \in \mc{B}_2(\pi)$ let us write
\begin{equation}\label{20210125_09:42}
\Phi(h) = \frac{\int_{-\infty}^{\infty} |h(x)|^2 \left( \frac{\sin \pi x}{\pi x}\right)^2\,\dx}{\int_{-\infty}^{\infty} |h(x)|^2 \,\dx}.
\end{equation}
For instance, for $h(x) = \frac{\sin \pi x}{\pi x}$, we have $\Phi(h) = \tfrac23$. Let $0 \neq f \in \mc{B}_2(\pi)$ be a maximizer for \eqref{20210125_09:42}, normalized so that $\|f\|_2 =1$. That is,
$$\Phi(f) = 1 - {\bf D}^2.$$
In what follows let us write $K_1(x) =  \left(\frac{\sin \pi x}{\pi x}\right)^2$, recalling our notation \eqref{20210125_09:36}. For any $h \in \mc{B}_2(\pi)$ with $\|h\|_2 =1$ and $h \perp f$, we have $\Phi(f + \varepsilon h) \leq \Phi(f)$ for any $\varepsilon \in \R$, with equality if $\varepsilon =0$. Therefore
\begin{align*}
0 = \frac{\partial}{\partial \varepsilon}\Phi(f + \varepsilon h)\Big|_{\varepsilon = 0} = 2 \,{\rm Re} \left(\int_{-\infty}^{\infty} f(x) \overline{h(x)} K_1(x)\,\dx\right).
\end{align*}
Similarly, for $\varepsilon \in \R$, 
\begin{align*}
0 = \frac{\partial}{\partial \varepsilon}\Phi(f + i\varepsilon h)\Big|_{\varepsilon = 0} = 2 \,{\rm Im} \left(\int_{-\infty}^{\infty} f(x) \overline{h(x)} K_1(x)\,\dx\right).
\end{align*}
We then conclude that 
\begin{align*}
0 = \int_{-\infty}^{\infty} f(x) K_1(x) \overline{h(x)}\, \dx = \int_{-\infty}^{\infty} \big(\widehat{f} *\widehat{K_1}\big)(\alpha)\,  \overline{\widehat{h}(\alpha)}\, \d\alpha.
\end{align*}
Since this holds for any $h \perp f$ in $\mc{B}_2(\pi)$, the function $\widehat{f}$ must verify the following Euler--Lagrange equation:
\begin{equation}\label{20210125_10:14}
\big(\widehat{f}*\widehat{K_1}\big) \,\chi_{[-\frac12,\frac12]}  = \eta \, \widehat{f},
\end{equation}
as functions in $L^2[-\frac12, \frac12]$, for some $\eta \in \C$. At this point observe that \eqref{20210125_10:14} yields 
\begin{equation*}
1 - {\bf D}^2 = \Phi(f) = \int_{-\infty}^{\infty} \big(\widehat{f} *\widehat{K_1}\big)(\alpha)\,  \overline{\widehat{f}(\alpha)}\, \d\alpha  = \eta.
\end{equation*}
Hence $\eta  \in \R$ and we have seen that $1 > \eta \geq \tfrac{2}{3}$.

\medskip

 Since the left-hand side of \eqref{20210125_10:14} is continuous in $[-\frac12, \frac12]$, we may assume that $\widehat{f}$ is continuous in $[-\frac12, \frac12]$ and hence
\begin{align}\label{20210125_10:15}
\big(\widehat{f}*\widehat{K_1}\big)(\alpha) = \int_{-\infty}^{\infty} \widehat{f}(\xi) \,\widehat{K_1}(\alpha - \xi)\,\d \xi = \eta \, \widehat{f}(\alpha)
\end{align}
for all $\alpha \in [-\frac12, \frac12]$. Since $\widehat{K_1}$ is a Lipschitz function, the integral in \eqref{20210125_10:15} (as a function of $\alpha$) is differentiable for all $\alpha \in (-\frac12, \frac12)$. Recalling that $\big(\widehat{K_1}\big)'(\alpha) = \chi_{(-1,0)}(\alpha) - \chi_{(0,1)}(\alpha)$, and that ${\rm supp}(\widehat{f}) \subset [-\frac12,\frac12]$, we have
\begin{align}\label{20210125_10:16}
\ \ \ \ \ \ \ \ \ \ - \int_{-\frac12}^{\alpha} \widehat{f}(\xi)\,\d\xi + \int_{\alpha}^{\frac12} \widehat{f}(\xi)\,\d\xi = \eta \, \big(\widehat{f}\big)'(\alpha) \ \ \ \ \ \ \ \ \ \big(\alpha \in \big(-\tfrac12, \tfrac12\big)\big).
\end{align}
The left-hand side of \eqref{20210125_10:16} is again differentiable in $\alpha$, and an application of the fundamental theorem of calculus now yields 
\begin{align*}
-2 \widehat{f}(\alpha) = \eta \, \big(\widehat{f}\big)''(\alpha) \ \ \ \ \ \ \ \ \ \big(\alpha \in \big(-\tfrac12, \tfrac12\big)\big).
\end{align*}
The general solution of this linear differential equation is  
\begin{align*}
\widehat{f}(\alpha) = \left(A \,e^{i \alpha \sqrt{2/\eta} } + B \,e^{-i \alpha\sqrt{2/\eta}}\right) \chi_{\big(-\frac12, \frac12\big)}(\alpha)\,,
\end{align*}
where $A, B \in \C$. Plugging this back into \eqref{20210125_10:16} we find the relation
\begin{equation*}
\cos\left( \frac{1}{\sqrt{2 \eta}}\right) (A- B) = 0.
\end{equation*}
Since $\eta \geq \tfrac23 > \tfrac{2}{\pi^2}$, we have $\cos(1/\sqrt{2\eta}) >0$ and therefore $A = B \neq 0$. Evaluating \eqref{20210125_10:15} at $\alpha =0$ we arrive at the condition
\begin{align}\label{20210125_14:20}
\left( \frac{1}{\sqrt{2 \eta}}\right) \tan\left( \frac{1}{\sqrt{2 \eta}}\right) = 1\,,
\end{align}
that determines our $\eta$ uniquely ($\eta = 0.67551\ldots$). Finally, the normalization $\|\widehat{f}\|_{L^2[-\frac12,\frac12]} =1$ together with \eqref{20210125_14:20} yields the value $|A| = \big((2\eta +1)/(8 \eta +2)\big)^{1/2}$. 

\medskip

In sum, our extremal function is unique (up to multiplication by a complex number) and its Fourier transform, with the substitution $\theta = (2 \pi^2 \eta)^{-1/2} $, is given by
\begin{align*}
\widehat{f}(\alpha) = 2A \cos\left( 2\pi \theta \, \alpha  \right) \,\chi_{\big(-\frac12, \frac12\big)}(\alpha),
\end{align*}
which, by Fourier inversion, leads us to \eqref{20210113_17:11}.

\section*{Acknowledgments}
\noindent Part of this paper was written while A.C. was a Visiting Researcher in Department of Mathematics at the University of Mississippi. He is grateful for their kind hospitality. We thank Oscar Quesada-Herrera for the design and implementation of the search algorithms in \S \ref{Sub_Pf_Thm1} and for the numerical computation of the constant in \eqref{20190705_11:49am}. We are also thankful to Jonathan Bober for an independent numerical computation of the constant in \eqref{20190705_11:49am}, and to Dan Goldston and Mateus Sousa for some helpful comments on an early draft of the paper. E.C. acknowledges support from FAPERJ - Brazil.  V.C. acknowledges support from an AMS-Simons Travel Grant and a Simons Foundation Collaboration Grant for Mathematicians. A.C. was supported by FAPERJ - Brazil and by Grant $275113$ of the Research Council of Norway. M.B.M. was supported in part by a Simons Foundation Collaboration Grant for Mathematicians.

\end{document}